\documentclass[sn-mathphys]{sn-jnl}% Math and Physical Sciences Reference Style
\jyear{2021}%

\renewcommand{\w}{\omega}
\newcommand{\ds}{\displaystyle}
\def \sgn {{\rm sgn}}

\newcommand{\tomega}{\tilde{\omega}}
\newcommand{\hatomega}{\breve{\omega}}
\newcommand{\tzeta}{\tilde{\zeta}}
\newcommand{\tvarepsilon}{\tilde{\varepsilon}}
\newcommand{\hatzeta}{\breve{\zeta}}
\newcommand{\barlambda}{\bar{\lambda}}
\newcommand{\bareta}{\bar{\eta}}
\theoremstyle{thmstyleone}%
\newtheorem{theorem}{Theorem}[section]%  meant for continuous numbers
\newtheorem{proposition}[theorem]{Proposition}% 

\theoremstyle{thmstyletwo}%
\newtheorem{remark}{Remark}%

\theoremstyle{thmstylethree}%
\newtheorem{definition}{Definition}[section]%
\newtheorem{lemma}{Lemma}[section]

\usepackage{amsmath,mathtools,amsthm}
\newcommand{\alert}[1]{{\color{black} #1}}	%red
\newcommand{\alertchieu}[1]{{\color{black} #1}}
\newcommand{\alertsection}[1]{{\color{black} #1}}
\newcommand{\tblue}[1]{{\color{black} #1}} %green
\usepackage{enumitem}

 \DeclarePairedDelimiter\abs{\lvert}{\rvert}
 \newcommand{\bftab}{\fontseries{b}\selectfont}
\begin{document}

%\title[Unified Smoothing Approach for Best Hyperparameter Selection Problem Using a Bilevel Optimization Strategy]{Unified Smoothing Approach for Best Hyperparameter Selection Problem Using a Bilevel Optimization Strategy}
%\title[Unified Smoothing Bilevel Approach for Hyperparameter Selection]{Unified Smoothing Approach for Best Hyperparameter Selection Problem Using a Bilevel Optimization Strategy\footnote[2]{The authors contributed equally to this work.}}

\title[Unified Smoothing Bilevel Approach for Hyperparameter Selection]{Unified Smoothing Approach for Best Hyperparameter Selection Problem Using a Bilevel Optimization Strategy}
%%=============================================================%%

\author[1]{\fnm{Jan Harold} \sur{Alcantara}}\email{\color{black}{janharold.alcantara@riken.jp}}

\author[2]{\fnm{Chieu Thanh} \sur{Nguyen}}\email{\color{black}{ntchieu@vnua.edu.vn}}

\author[1,3]{\fnm{Takayuki} \sur{Okuno}}\email{\color{black}{takayuki-okuno@st.seikei.ac.jp}}

\author[1,4]{\fnm{Akiko} \sur{Takeda}}\email{\color{black}{takeda@mist.i.u-tokyo.ac.jp}}

\author*[5]{\fnm{Jein-Shan} \sur{Chen}}\email{ \color{black}{jschen@math.ntnu.edu.tw}}

\affil[1]{\orgdiv{Center for Advanced Intelligence Project}, \orgname{RIKEN}, 
	\orgaddress{\city{Tokyo}, \postcode{103-0027}, \state{Japan}}}

\affil[2]{\orgdiv{Department of Mathematics}, \orgname{Faculty of Information 
Technology, Vietnam National  University of Agriculture}, 
\orgaddress{\city{Hanoi}, \postcode{131000}, \state{Vietnam}}}

\affil[3]{\orgdiv{Faculty of Science and Technology}, \orgname{Seikei 
University}, 
\orgaddress{\city{Tokyo}, \postcode{180-8633}, \state{Japan}}}

\affil[4]{\orgdiv{Department of Mathematical Informatics}, \orgname{Graduate 
School of Information Science and Technology, The University of Tokyo}, 
\orgaddress{\city{Tokyo}, \postcode{113-8656}, \state{Japan}} }

\affil*[5]{\orgdiv{Department of Mathematics}, \orgname{National Taiwan Normal University}, \orgaddress{\city{Taipei}, \postcode{11677}, \state{Taiwan}}}


\abstract{Strongly motivated from use in various fields including machine learning,
the methodology of sparse optimization has been developed intensively so far.
Especially,  the recent advance of algorithms for solving problems with nonsmooth regularizers  is remarkable.  However,  those algorithms suppose that weight parameters of regularizers, called hyperparameters hereafter, are pre-fixed,  and it is a crucial matter how the best hyperparameter should be selected.
In this paper, we focus on the hyperparameter selection of regularizers related to the $\ell_p$ function with $0<p\le 1$
and \alert{apply a} bilevel programming strategy, wherein we need to solve a 
bilevel problem, whose lower-level problem is
\alert{nonsmooth}, possibly nonconvex and non-Lipschitz. Recently, for solving 
a bilevel problem for hyperparameter selection of the pure $\ell_p\ (0<p \le 
1)$ regularizer
Okuno et al. discovered new necessary optimality conditions, called SB(scaled 
bilevel)-KKT conditions,
and further proposed a smoothing-type algorithm using a certain
smoothing function. \alert{However, this optimality measure is loose in the sense 
that there could be many points that satisfy the SB-KKT conditions.}

\alert{In this work, we propose new bilevel KKT conditions, which are new 
necessary optimality conditions tighter than the ones proposed by Okuno et al. 
Moreover, we propose} a unified smoothing approach using smoothing functions 
that belong to the Chen-Mangasarian class, and then prove that  generated 
iteration points accumulate at \alert{bilevel KKT points }under milder 
constraint qualifications.
Another contribution is that our approach and analysis are applicable to a 
wider class of regularizers.
Numerical comparisons demonstrate which smoothing functions work well for hyperparameter optimization via bilevel
optimization approach.
%Numerical comparisons are presented that demonstrate which smoothing functions work well for hyperparameter optimization via bilevel
%optimization approach.
}

%%================================%%




\maketitle

\section{Introduction}\label{sec1}
A learning algorithm in machine learning usually involves solving the unconstrained optimization problem 
	\begin{equation}\label{lower-levelproblem0}
	\min _{\omega \in \Re^n} g(\omega) + \sum_{i=1}^r \lambda_i R_i(\omega),
	\end{equation}
where $\lambda = (\lambda_1, \dots, \lambda_r)$ is called a hyperparameter, 
whose value is decided prior to implementation of the learning algorithm. Here, 
$R_i, g :  \Re^n \to \Re$, $i=2,\dots, r$ are twice continuously differentiable 
functions, and 

\begin{equation}
R_1(\omega) \coloneqq \sum_{i=1}^n \psi(\lvert\omega_i\rvert^p) \quad (0 < p 
\leq 1)
\label{eq:R1}
\end{equation}
with $\psi$ satisfying the following assumption:

\medskip 

\alert{\noindent \textbf{Assumption~(A).}} $\psi: [0, \infty) \to \Re$ is twice 
continuously differentiable on $[0, \infty)$ and there exist two positive 
constants $\alpha, \beta$ such that
$0< \psi'(t)\leq \alpha$ and $-\beta \leq \psi''(t)\leq 0$ for all $t \in [0, \infty)$.

\medskip 
 In this manuscript, we make \alert{Assumption~(A)} our blanket assumption on 
 $\psi$. It is well-known that the function $R_1(\omega)$ is nonsmooth, 
 nonconvex, even non-Lipschitz when $p\in (0,1)$. There are many penalty 
 functions often used in statistics and signal reconstruction satisfying 
 Assumption~(A) (see Appendix~\ref{app:penaltyfunctions}). 
 
 \medskip 
 
 For notation purposes, we denote
 \[
 G(\omega, \bar{\lambda}) \coloneqq g(\omega) + \alert{ \barlambda^T 
 \bar{R}(w)},
 \]
 with $\bar{\lambda} \coloneqq (\lambda_2, \dots, \lambda_r)^T \in 
 \Re^{r-1}$ \alert{and $\bar{R}: \Re^n \to \Re^{r-1}$ given by $\bar{R}(w) 
 \coloneqq 
 (R_2(w),\dots, 
 R_r(w))^T$}. Then problem \eqref{lower-levelproblem0} can be rewritten as 

	\begin{equation}\label{lower-levelproblem}
	\min_{\omega\in \Re^n}G(\omega, \bar{\lambda}) +  \lambda_1 R_1(\omega).
	\end{equation} 

 \medskip 
 
 The problem of finding \alert{the} optimal values of the hyperparameters for 
 \eqref{lower-levelproblem} can 
 be accomplished using grid search and Bayesian optimization 
 \cite{BB12,WHZMF16}. This paper, on the other hand, is devoted to a bilevel 
 optimization strategy to find the best hyperparameter. In particular, we focus 
 on the bilevel nonsmooth programming problem
 \begin{equation} \label{bileveloptimization}
 \begin{array}{cc}
  \displaystyle{\min_{\omega^*_{\lambda},\lambda}} & f(\omega^*_{\lambda}) \\
 {\rm s.t} & \omega^*_{\lambda} \in \underset{\omega\in \Re^n}{\mbox{argmin}}\; G(\omega, \bar{\lambda}) + \lambda_1 R_1(\omega) \\
 & \alert{(\lambda_1,\bar{\lambda}) \in \Omega_{\epsilon} \subset \Re^r},
 \end{array}
 \end{equation}
 where $f: \Re^n \to \Re$ is continuously differentiable \alert{and 
 	\begin{equation}
 		\Omega_{\epsilon} \coloneqq \{ (\lambda_1,\bar{\lambda})\in \Re\times 
 		\Re^{r-1} : 
 		\lambda_1\geq \epsilon , \, \barlambda \geq 0\},
 		\label{eq:Omega_eps}
 	\end{equation}
 	for some small parameter $\epsilon > 0$. }
 Problem \eqref{lower-levelproblem} that appears in the constraint set of 
 \eqref{bileveloptimization} is called the lower-level problem, and the 
 minimization of $f$ is called the upper-level problem. \alert{Note that in the 
 interest of obtaining sparse models, we impose a strict positive lower bound 
 for the parameter $\lambda_1$ corresponding to the sparsity-promoting 
 regularizer $R_1$.}

\medskip 
 
  Bilevel optimization problems were introduced by Bracken and McGill 
  \cite{BM73}. The reader is referred to \cite{D03,CMS07,SMD18} for a survey of 
  methods for solving the bilevel optimization problem and their applications. 
  Efforts have been put forth by many researchers in the past few decades to 
  use bilevel optimization strategy to the problem of finding the best 
  hyperparameter values. In particular, \cite{BHJKP06,BKJP08} focused on a 
  bilevel support-vector regression (SVR) problem where the lower-level 
  optimization problem is cast as a convex quadratic program. 
  The authors in \cite{MBB09,MBB11} proposed a bilevel cross-validation program 
  for 
  support-vector machine (SVM), where the upper-level problem is convex and 
  nonsmooth, while the lower-level problem is differentiable. \cite{ORBP16} 
  used gradient-based methods for the bilevel optimization problem with 
  nonsmooth convex lower-level problem (for example, sparse models based on the 
  $\ell_1$-norm). However, \cite{BHJKP06,MBB09,ORBP16} only proposed the 
  algorithms to solve the bilevel optimization without providing any 
  theoretical analysis. \cite{TBP20} constructed the hyperparameter 
  optimization problem through $K$-fold cross-validation as a bilevel 
  optimization problem with LASSO regression and an $\ell_1$-norm 
  support-vector machine (SVM) in the lower-level problem. They used parametric 
  programming theory to reformulate the bilevel optimization problem as a 
  single level problem, which is called the bilevel and parametric optimization 
  approach to hyperparameter optimization (HY-POP). Similar to that in 
  \cite{TBP20}, the authors only provided the numerical experiments to show the 
  efficiency of HY-POP without any theoretical analysis. \cite{KP13} considered 
  bilevel optimization problems for variational image denoising models, where 
  the upper-level problem is smooth while the lower-level problem is the 
  $\ell_p$ regularizer with $p=\frac{1}{2},1,2$. They proposed semismooth 
  Newton method for solving the bilevel optimization problem including the 
  $\ell_2$-norm and the $\ell_1$-norm. Especially, they only provided numerical 
  experiments for the $\ell_{\frac{1}{2}}$-norm and leave the theoretical 
  analysis for nonconvex $\ell_{\frac{1}{2}}$-norm to future work. 
  Nevertheless, they showed that the $\ell_{\frac{1}{2}}$-norm has better 
  denoising performance than the $\ell_1$-norm. Recently, \cite{OTKW21} 
  considered the bilevel program \eqref{bileveloptimization} with the function 
  $R_1(\omega) \coloneqq \|\w\|_p^p = \sum_{i=1}^n \lvert\omega_i\rvert^p$ $(0 
  < p \leq 1)$ 
  (i.e. 
  the $\ell_p$-regularizer) \alert{and $\epsilon = 0$} by employing a smoothing 
  method 
  via the twice continuously differentiable function 
 \begin{equation}
  \varphi _{\mu } (\omega) = \sum _{i=1}^n (\omega _i^2 + \mu ^2)^{\frac{p}{2}} \label{R1_smooth_OTKW21}
 \end{equation}
 as a smooth approximation of $R_1$. Using such a smoothing function, problem 
 \eqref{bileveloptimization} can be approximated by a smooth bilevel program, 
 which then allows for use of several optimization techniques that normally 
 require differentiability. Thus, they established the convergence analysis for 
 the $\ell_p$-norm with $p\in (0,1]$.
 
 \medskip 
 
\alert{The following are the main theoretical contributions of our present work:
 \begin{enumerate}
 	\item[(I)]  First, we propose bilevel KKT conditions (BKKT conditions for 
 	short) for 
 		problem \eqref{bileveloptimization}, which are new necessary 
 		optimality conditions for the relaxation of 
 		\eqref{bileveloptimization} obtained by replacing its lower-level 
 		optimization problem by the corresponding first order necessary 
 		conditions in terms of generalized subdifferentials (see Section 
 		\ref{sec:subdiff}), that is, 
 		 \begin{equation} \label{first-orderoptimality}
 		 \begin{array}{cc}
 		 \min\limits_{\omega,\lambda} & f(\omega) \\
 		 {\rm s.t} & 0 \in \partial_{\omega}(G(\omega, \bar{\lambda}) +  
 		 \lambda_1 
 		 R_1(\omega)) \\
 		 &  \alert{(\lambda_1,\barlambda)\in \Omega_{\epsilon} }.
 		 \end{array}
 		 \end{equation} 
 	Our proposed BKKT conditions are notably tighter than the scaled bilevel 
 	KKT conditions (SB-KKT conditions for short)
 	discovered in \cite{OTKW21}. As a special case, when $p=1$ and the 
 	functions 
 	$f$, $g$ and $R_i$ ($i=1,\dots, r$) are all convex functions, the proposed 
 	BKKT 
 	conditions are necessary optimality conditions for the original bilevel 
 	problem \eqref{bileveloptimization}. 
 	\item[(II)] Second, we consider a general framework for 
 	constructing smoothing functions for $R_1$ given by \eqref{eq:R1}, where 
 	the associated $\psi$ is 
 	any function that satisfies Assumption~(A) and the absolute value mapping
 	is smoothly approximated by a function generated via density functions, as 
 	inspired by the smoothing technique for plus functions by Chen and 
 	Mangasarian \cite{CM96}. Based on this approach, we propose a 
 	smoothing algorithm and prove its convergence to BKKT points by 
 	utilizing only some information on the generating density function. That 
 	is, we do not rely on a specific formula of a smoothing function, and 
 	therefore our framework provides a unified theory for a class of smoothing 
 	algorithms 
 	for \eqref{bileveloptimization}. Indeed, one novelty of this work is our 
 	unified convergence analysis that solely depends on density functions. 
 	Along with these, we only suppose weaker algorithmic assumptions and 
 	constraint 
 	qualifications, as opposed to the specific model and algorithm considered 
 	in \cite{OTKW21}. Finally, in connection with contribution (I) described 
 	above, we obtain stronger results since we establish
 	convergence to	BKKT points, which are tighter necessary conditions than 
 	SB-KKT conditions. 
 \end{enumerate}

\medskip 
The SB-KKT conditions proposed in \cite{OTKW21} for problem \eqref{first-orderoptimality} with $\epsilon = 0 $ are more 
loose than our proposed BKKT conditions as mentioned in (I). Consequently, 
we provide a better optimality measure for the relaxation 
\eqref{first-orderoptimality} of the bilevel program 
\eqref{bileveloptimization}. In fact, when $p=1$, the 
SB-KKT conditions proposed in \cite{OTKW21} are not even necessary conditions 
for the relaxed problem \eqref{first-orderoptimality}, but for another 
relaxation which has a larger feasible region (see model 
\eqref{scaledone-levelproblem} and Proposition 
\ref{prop:scaled-lower-level}). Hence, our proposed BKKT 
conditions provide a significant improvement over the prior work. 

Moreover, under an appropriate assumption on the algorithm iterates (see Remark 
\ref{rem:liminf}), our 
convergence analysis significantly generalizes the existing technique of 
\cite{OTKW21} that only holds for 
the case when $\epsilon = 0$, $\psi(t) \equiv t$, and the 
function $\varphi_{\mu }$ in \eqref{R1_smooth_OTKW21} is used to smoothly 
approximate the $\ell_p$ norm in 
\eqref{bileveloptimization}. In the said work, the formula of the 
smoothing function \eqref{R1_smooth_OTKW21} is fully exploited to derive 
important inequalities specific to \eqref{R1_smooth_OTKW21}, and to obtain 
fundamental lemmas for establishing global 
subsequential convergence (see, for instance, \cite[Lemma 7, Proposition 8, and 
the proof 
of Theorem 5]{OTKW21}). 
%For 
%instance, \cite[Lemma 7]{OTKW21} provides the relationship between the rates 
%of 
%decay to zero of the smoothing parameter $\mu$ and the sequence of iterates 
%$\{ 
%\omega^k\}$ produced by the smoothing algorithm, and inequalities particular 
%to 
%the function \eqref{R1_smooth_OTKW21} are used to obtain the result. 
Indeed, 
the lines of arguments 
used to establish the aforementioned results
are only applicable to the chosen smoothing function \eqref{R1_smooth_OTKW21}. 
It should be noted that extension to a wider class of regularizers $R_1$ given 
by \eqref{eq:R1} 
with an arbitrary smooth approximation of the absolute value function is not 
trivial and requires more subtle arguments. To this 
end, the present work provides a unified analysis that rather derives 
alternative fundamental lemmas and properties, using arguments that do not rely 
on the specific formula of a smoothing function, but only some analytic 
properties of a density function generating the smoothing function. In turn, 
other important 
contributions of our work also involve the flexibility of our
algorithm in terms of the smoothing functions used and its applicability to a 
considerably wider class of 
regularizers for the hyperparameter optimization problem, that also come with
convergence guarantees under less restrictive constraint qualifications and 
weaker algorithmic assumptions. }
 \medskip 
 
 From a practical point of view, the choice of smoothing functions is critical 
 in achieving successful simulations with fast convergence rates. 
% \alert{Towards this end, we present a practical implementation of our 
% algorithm that uses a semismooth Newton method for the bilevel KKT system.} 
We  compare 
 the numerical performance of six smoothing functions generated via Chen and 
 Mangasarian's method \cite{CM96} to determine which function is more suitable 
 for our smoothing approach. \alert{Our proposed algorithm involves the use of a semismooth Newton method to solve a sequence of bilevel KKT systems, thereby significantly improving upon the methodology proposed in \cite{OTKW21}. As a result, one significant finding from our 
 numerical experience indicates that some smoothing functions result to a faster 
 smoothing algorithm that obtains sparse models with lower validation and test 
 errors, and thus giving insights on which smoothing function can work well 
 with the proposed strategy. }

 \medskip

 This paper is organized as follows: \alert{In Section 2, we review some fundamental 
 concepts in analysis and a brief review of the method proposed in \cite{CM96} 
 to construct smoothing functions of the plus function by means of density 
 functions. This will serve as our basis to construct smoothing functions for 
 $R_1(\omega)$, and our theoretical analysis will all be dependent on the 
 density function. In Section 3, we recall the SB-KKT conditions utilized in 
 \cite{OTKW21}, and then propose our BKKT conditions. In Section 4, we present 
 our smoothing algorithm along with its convergence analysis. In Section 5,  we 
 compare the numerical performance of different smoothing functions generated 
 from different density functions in solving \eqref{bileveloptimization}.}

 \medskip

 Throughout this paper, we denote the vector $\omega\in \Re^n$ by 
 $\omega=(\omega_1, \dots, \omega_n)^T$. We let 
 $\lvert\omega\rvert\coloneqq(\lvert\omega_1\rvert, \dots, 
 \lvert\omega_n\rvert)^T$, 
 and $\lvert\omega\rvert^p\coloneqq(\lvert\omega_1\rvert^p, \dots, 
 \lvert\alert{\omega_n}\rvert^p)^T$. \alert{We define $I(\omega) \coloneqq \{j 
 \in 
 \{1,2,\dots,n\} ~\lvert~ \omega_j = 0\}$ for any $\w\in\Re^n$. The Hadamard 
 product of two vectors 
 $\omega\in \Re^n$ and 
 $\hatomega \in 
 \Re^n$ is denoted by $\omega \odot \hatomega \coloneqq 
 (\omega_1 \hatomega_1, \dots, \omega_n \hatomega_n)^T$}.  We define 
 the 
 $\sgn$ function as sgn$(t)=1$  $(t>0)$, sgn$(t)$ = 0 $(t=0)$, and sgn$(t)=-1$ 
 $(t<0)$. For a differentiable function $f:\Re^n\to \Re$, we denote the 
 gradient function of $f$ by $\nabla f$ with $\nabla f(\omega) \coloneqq 
 (\frac{\partial f(\omega)}{\partial\omega_1}, \dots, \frac{\partial 
 f(\omega)}{\partial\omega_n}) ^T \in \Re^n$ and if $f$ is twice 
 differentiable, we denote the Hessian of $f$ by $\nabla^2 f$ with $\nabla^2 
 f(\omega) \coloneqq \left(\frac{\partial^2 
 f(\omega)}{\partial\omega_i\partial\omega_j}\right)_{1\leq i,j\leq n} \in 
 \Re^{n\times n}$.

%---------------------------------------------------------------------------------------------- Section 2
 \section{Preliminaries}
\alert{  We review some important concepts in nonsmooth analysis. We also recall the 
  method of Chen and Mangasarian to construct smoothing functions for the plus 
  function, and discuss how to use this to obtain a smoothing function for the 
  absolute value function. }
  
%--------------------------------------------------------------------------------- Subsection 2.1
 \subsection{Some concepts in analysis}\label{sec:subdiff}
 The following facts can be found in the books of Rockafellar and Wet \cite{RW09}.

%----------------------------------------------------------------- Definition 2.1
 \begin{definition}\cite[Definition 8.3]{RW09}\label{generalsubgradient}
 Let $f: \Re^n \to \Re\cup \{\infty\}$ be a proper function. For vectors $v\in \Re^n$ and $\bar{x}\in \Re^n$, one say that
 \begin{itemize}
 \item [1.]  $v$ is a regular subgradient of $f$ at $\bar{x}$, written $v \in \hat{\partial}f(\bar{x})$, if 
 \[
 f(x) \geq f(\bar{x}) + v^T(x - \bar{x}) + o(\|x - \bar{x}\|).
 \]
 \item [2.] $v$ is a general subgradient of $f$ at $\bar{x}$, written $v \in \partial f(\bar{x})$, if there are sequences $\{x^{\nu}\}\subseteq \Re^n$ and $\{v^{\nu}\}\subseteq \Re^n$ such that
 \[
 \lim\limits_{\nu\to\infty}x^{\nu} = \bar{x} \; \mbox{and}\; v^{\nu} \in \hat{\partial}f(x^{\nu}) \; \mbox{with} \; \lim\limits_{\nu\to\infty}v^{\nu} = v.
 \]
 \end{itemize}
 \end{definition}
 Note that a regular subgradient of $f$ at $\bar{x}$ is also called a Fr\'{e}chet subgradient of $f$ at $\bar{x}$ (see in \cite{K03}). Moreover, if $f$ is a proper and convex function, the regular subgradient of $f$ coincides \tblue{with} the subgradient of $f$ in the sense of convex analysis (see in \cite[Proposition 8.12]{RW09}).
%--------------------------------------------------------------------------- Proposition 2.1
 \begin{proposition}\cite[Theorem 8.6]{RW09}
 For a function $f: \Re^n \to \Re\cup \{\infty\}$ and a point $\bar{x}$ where $f$ is finite, the subgradient sets $\hat{\partial}f(\bar{x})$ and $\partial f(\bar{x})$ are closed, with
 $\hat{\partial}f(\bar{x})$ convex and $\hat{\partial}f(\bar{x}) \subset \partial f(\bar{x})$.
 \end{proposition}

%------------------------------------------------------------------------------ Proposition 2.2
 \begin{proposition}\cite[Theorem 10.1]{RW09}\label{prop22}
 If a proper function $f: \Re^n \to \Re\cup \{\infty\}$ has a local minimum at $\bar{x}$, then $0 \in \hat{\partial}f(\bar{x}) \subset \partial f(\bar{x})$.
 \end{proposition}

%----------------------------------------------------------------------------------------------
% Subsection 2.2
 \subsection{Smoothing functions of $\lvert x\rvert$ via density functions}\label{subsec:smoothingfunctions}

 We recall the general definition of a smoothing function. 
 \begin{definition}\cite[Definition 1]{C12}\label{def-smoothingfunction}
 Let $h: \Re^n \to \Re$ be a continuous function. We say that $\phi: \Re_{++} \times \Re^n \to \Re$ is a smoothing function  
 of $h$ if it satisfies the following:
 \begin{description}
 \item[(i)] $\phi(\mu, \cdot)$ is continuously differentiable for any $\mu>0$;

 \item[(ii)] $\ds \lim_{w\to z,\mu\downarrow 0}\phi(\mu,w) = h(z)$ for any $z\in \Re^n$.
 \end{description} 
 \end{definition}

 \alert{To construct a smoothing function for the absolute value function, we 
 briefly recall from \cite{C12,CM96} that the plus function $(x)_+ = \max \{ 
 x,0\}$ for $x\in \Re$ can be smoothly approximated by 
  \begin{equation}\label{smoothplus}
 	\hat{\phi}(\mu,x) = \int_{-\infty}^{+\infty}(x-t)_+\hat{t}(\mu,t)dt = 
 	\int_{-\infty}^{x}(x-t)\hat{t}(\mu,t)dt,
 \end{equation}
where $\hat{t}(\mu,t) \coloneqq \frac{1}{\mu}\rho\left(\frac{t}{\mu}\right),$ 
and 
$\rho:\Re\to \Re_+$ is a a piecewise continuous density 
function\footnote{\alert{That is, $\rho$ is a nonnegative function whose 
integral over $\Re$ is 1. Consequently, it is easy to see that 
$\hat{t}(\mu,x)\to \delta(x)$ as $\mu\to 0$ for all $x\in\Re$, where $\delta$ 
is the Dirac delta function provided that $\rho(0)>0$.}} that satisfies
 \begin{equation}\label{kernelcondition}
 \rho(x) = \rho(-x)\ \mbox{and}\ \kappa \coloneqq 
 \int_{-\infty}^{+\infty}\lvert x\rvert\rho(x)dx < +\infty.
 \end{equation}

Using the fact that $\lvert x\rvert = (x)_+  + (-x)_+$, we obtain a smoothing function for the absolute value function as follows:
 \begin{equation}\label{smoothabsolutevalue}
 \phi(\mu,x) \coloneqq \hat{\phi}(\mu,x) + \hat{\phi}(\mu,-x) = 
 \int_{-\infty}^{+\infty}\lvert x-t\rvert\hat{t}(\mu,t)dt.
 \end{equation}

 In a manner similar to \cite[Proposition 2.2]{CM96},} we have the following properties of $\phi(\mu,x)$.

%----------------------------------------------------------------------------- Proposition 3.2
 \begin{proposition}\label{properties-smoothabsolutevalue}
 Suppose that $\phi(\mu,x)$ is defined as in (\ref{smoothabsolutevalue}). Then, for a fixed $\mu>0$, we have
 \begin{description}
 \item [(a)] $\phi(\mu,\cdot)$ is continuously differentiable.
 \item [(b)] $0\leq \phi(\mu,x) -  \lvert x\rvert \leq \kappa\mu$ for all $x\in 
 \Re$ and $\mu>0$, where the constant $\kappa>0$ is defined in 
 \eqref{kernelcondition}.
 \item [(c)] $\phi'(\mu,x)$ is bounded satisfying
             $-1 \leq \phi'(\mu,x) \leq 1$ for all $x\in \Re$, $\mu>0$.
 \end{description}
 \end{proposition}
 From Proposition~\ref{properties-smoothabsolutevalue}, given any sequence 
 $\{(x^k,\mu_k\}\subset \Re$ such that $x^k\to x\in \Re$ and $\mu_k \downarrow 
 0$,  we have 
  \begin{equation*}\label{sequence-smoothabsolute}
 \lim_{k\to \infty}\phi(\mu_k,x^k) = \lvert x\rvert , \quad \forall x\in \Re 
 \end{equation*}
 and 
 %% \begin{equation}\label{bounded-sequence-smoothabsolute}
 %% \left\{\lim_{x^k\to x,\mu_k\downarrow 0}\phi'(\mu_k,x^k) \right\} \subseteq 
 %% \left\{
 %% \begin{array}{ccc}
 %% \{ 1\} & \mbox{if}\ x>0\\
 %% \{ -1\} & \mbox{if}\ x<0\\
 %% \left[-1,1\right] & \mbox{if}\ x=0.
 %% \end{array}
 %% \right.
 %% \end{equation}
  \begin{equation}\label{bounded-sequence-smoothabsolute}
 \lim_{k\to\infty}\phi'(\mu_k,x^k) = \sgn (x) \quad \forall x\neq 0.
 \end{equation} 
We also have from Proposition~\ref{properties-smoothabsolutevalue}(c) that 
subsequential limits of the sequence on the left-hand side of 
$\{\phi'(\mu_k,x^k)\}$ exist and belong to $[-1,1]$.
%  
% There are various choices of the density function $\rho$ and corresponding smoothing function $\phi$. Many instances will be presented in the section of numerical experiments. 

 \medskip

 \section{\alert{Necessary conditions}}
 
Using Proposition~\ref{prop22}, the first-order optimality condition for the 
lower-level problem \eqref{lower-levelproblem} is given by
\[
0 \in \partial_{\omega}(G(\omega^*, \bar{\lambda}) +  \lambda_1 R_1(\omega^*)),
\]
where $\partial_{\omega}(G(\omega^*, \bar{\lambda})+\lambda_1 R_1(\omega ^*) ) 
$ is the general subgradient with respect to $\omega$ of $G(\omega, 
\bar{\lambda}) + \lambda_1 R_1 (\omega) $ at $\omega^*$. Then problem 
\eqref{bileveloptimization} can be transformed into the one-level 
problem given in \eqref{first-orderoptimality}.
%\begin{equation} \label{first-orderoptimality}
%\begin{array}{cc}
%\min\limits_{\omega,\lambda} & f(\omega) \\
%{\rm s.t} & 0 \in \partial_{\omega}(G(\omega, \bar{\lambda}) +  \lambda_1 
%R_1(\omega)) \\
%&  \alert{(\lambda_1,\barlambda)\in \Omega_{\epsilon} }.
%\end{array}
%\end{equation} 

\subsection{Scaled Bilevel KKT Conditions}
\alert{In \cite{OTKW21}, a smooth version of the lower-level problem of 
\eqref{first-orderoptimality} when $R_1(\w) = \|\w\|_p^p$ was proposed using 
the scaled first-order 
necessary condition initiated by Chen, Xu and Ye \cite{CXY10} for non-Lipschitz 
continuous functions. Since 
the function $G(\omega, \bar{\lambda}) +  \lambda_1 \|\w\|_p^p$ may be 
non-Lipschitz, the scaled first-order necessary 
condition for the lower-level problem \eqref{lower-levelproblem} proposed in
\cite{OTKW21}, which was adapted from \cite{CXY10}, can be extended to our 
setting with 
$R_1$ given by \eqref{eq:R1}, as in Definition~\ref{defn:scaled_first_order}. 
In particular, when $\psi (t) \equiv t$, the following definition reduces to 
the scaled first-order necessary condition given in \cite{OTKW21}.}

%--------------------------------------------------------------------- 
%Definition 2.1
\begin{definition}\label{defn:scaled_first_order}
	We say that $\omega^*$ satisfies the scaled first-order necessary condition 
	of \eqref{lower-levelproblem} if
	\begin{equation*}\label{scaledfirst-ordercondition}
	W_*\nabla_{\omega}G(\omega^*,\bar{\lambda}) + p\lambda_1\lvert W_*\rvert^p 
	\psi'(\lvert\omega^*\rvert^p) = 0,
	\end{equation*}
	where $W_* \coloneqq {\rm diag}(\omega^*)$, $\lvert W_*\rvert^p \coloneqq 
	{\rm 
	diag}(\lvert\omega^*\rvert^p)$, and 
	\[
	\psi'(\lvert\omega^*\rvert^p) \coloneqq (\psi'(\lvert\omega_1^*\rvert^p), 
	\psi'(\lvert\omega_2^*\rvert^p), \dots, \psi'(\lvert\omega_n^*\rvert^p))^T.
	\]
\end{definition}
Using this scaled first-order necessary condition, \alert{one can} obtain the 
following 
one-level problem.
\begin{equation} \label{scaledone-levelproblem}
\begin{array}{cc}
\min\limits_{\omega,\lambda} & f(\omega) \\
{\rm s.t} & W\nabla_{\omega}G(\omega,\bar{\lambda}) + p\lambda_1 \lvert 
W\rvert^p \psi'(\lvert\omega\rvert^p) = 0 \\
& \lambda \alert{\in \Omega_{\epsilon}},
\end{array}
\end{equation}
where $W \coloneqq {\rm diag}(\omega)$, $\lvert W\rvert^p \coloneqq {\rm 
diag}(\lvert\omega\rvert^p)$, and
\[
\psi'(\lvert\omega\rvert^p) \coloneqq (\psi'(\lvert\omega_1\rvert^p), 
\psi'(\lvert\omega_2\rvert^p), \dots, \psi'(\lvert\omega_n\rvert^p))^T.
\]
\alert{Though this problem looks different from \eqref{first-orderoptimality}, 
the
	following proposition\footnote{The proof is analogous to \cite[Lemma 
		3]{OTKW21}.} indicates that the problems are indeed identical when 
		$p\in 
	(0,1)$. However, when $p=1$, the feasible 
	region of 
	problem \eqref{first-orderoptimality} is contained in that of 
	\eqref{scaledone-levelproblem}.}

%--------------------------------------------------------------------------- 
%Lemma 2.1
\begin{proposition}\label{prop:scaled-lower-level}
	For $\omega \in \Re^n$ and $\lambda \in \Re_+^r$, if $0 \in 
	\partial_{\omega}(G(\omega, \bar{\lambda}) +  \lambda_1 R_1(\omega))$, then 
	$W\nabla_{\omega}G(\omega,\bar{\lambda}) + p\lambda_1 \lvert W\rvert^p 
	\psi'(\lvert\omega\rvert^p) = 0$. In particular, when $p<1$, the converse 
	is also true.
\end{proposition}
%From the above proposition, we observe that the feasible region of 
%\eqref{scaledone-levelproblem} is larger than the feasible region of 
%\eqref{first-orderoptimality} when $p=1$. Moreover, the feasible regions of 
%problems \eqref{first-orderoptimality} and \eqref{scaledone-levelproblem} are 
%identical when $p<1$. 
\medskip 

\alert{Based on this scaling, one can extend the scaled bilevel KKT (SB-KKT) 
conditions proposed in \cite{OTKW21} to our general setting of 
\eqref{bileveloptimization} as in the following definition. }

\medskip

%--------------------------------------------------------------------------- 
%Definition 2.2
\begin{definition}\label{SBKKTpoint}
	We say that $(\omega^*,\lambda^*)\in \Re^n \times \Re^r$ is a scaled 
	bilevel Karush-Kuhn-Tucker (SB-KKT) point for problem 
	\eqref{bileveloptimization} if there exists a pair of vectors $(\zeta^*, 
	\eta^*) \in \Re^n \times \Re^r$ such that
	\begin{eqnarray}
	& W^2_*\nabla f(\omega^*) + H(\omega^*,\lambda^*)\zeta^* = 0, 
	\label{SBKKT1}\\
	& W_*\nabla_{\omega}G(\omega^*,\bar{\lambda}^*) + p\lambda^*_1\lvert 
	W_*\rvert^p \psi'(\lvert\omega^*\rvert^p) = 0, \label{SBKKT2}\\
	& p\sum_{j\notin I(\omega^*)}{\rm 
	sgn}(\omega^*_j)\lvert\omega_j^*\rvert^{p-1}\psi'(\lvert\omega_j^*\rvert^p) 
	\zeta^*_j = \eta^*_1, \label{SBKKT3}\\
	& \zeta^*_j = 0 \ (j \in I(\omega^*)), \label{SBKKT4}\\
	& \nabla R_j(\omega^*)^T \zeta ^*- \eta^*_j = 0 \ (j=2,3, \dots, r), 
	\label{SBKKT5} \\
	& \alert{\lambda^*- \epsilon e_1\geq 0, \quad 
	\eta^* \geq 0, 
	\quad 
	(\lambda^* - \epsilon e_1)^T \eta^* 
	= 0}, \label{SBKKT6} 
	\end{eqnarray}
	where $W_* \coloneqq {\rm diag}(\omega^*)$, and \alert{$e_1=(1,0,\dots, 
	0)^T\in 
	\Re^r$}. 
	Here, we write
	\[
	H(\omega,\lambda) = W^2\nabla^2_{\omega \omega }G(\omega,\bar{\lambda}) + 
	\lambda_1 p(p-1){\rm diag}(\lvert W\rvert^p\psi'(\lvert\omega\rvert^p)) + 
	\lambda_1 p^2{\rm diag}(\lvert W\rvert^{2p}\psi''(\lvert\omega\rvert^p))
	\]
	with $W \coloneqq {\rm diag}(\omega)$, $\lvert W\rvert^p \coloneqq {\rm 
	diag}(\lvert\omega\rvert^p)$, and $\lvert W\rvert^{2p} \coloneqq {\rm 
	diag}(\lvert\omega\rvert^{2p})$ for $\omega \in \Re^n$ and $\lambda \in 
	\Re^r$.
\end{definition}

\alert{The SB-KKT conditions are necessary optimality conditions for problem 
\eqref{scaledone-levelproblem} as asserted in the following result, whose proof 
is essentially similar to \cite[Theorem 2]{OTKW21}. }

%------------------------------------------------------------------ Proposition 
%2.1
\begin{proposition}\label{prop:sbkkt_is_necessary_for_relaxed}
	Let $(\omega^*, \lambda^*) \in \Re^n \times \Re^r$ be a local optimum of 
	\eqref{scaledone-levelproblem}. Then, $(\omega^*, \lambda^*)$ together with
	%some vectors $\{\zeta^*, \eta^*\} \in \Re^n\times \Re^r$
	some pair of vectors $(\zeta^*, \eta^*)\in \Re^n\times \Re^r$
	satisfies the SB-KKT conditions \eqref{SBKKT1}-\eqref{SBKKT6} under an 
	appropriate constraint qualification concerning the constraints 
	$\frac{\partial G(\omega, \bar{\lambda})}{\partial \omega_j} + p {\rm 
	sgn}(\omega_j)\lambda_1\lvert\omega_j\rvert^{p-1}\psi'(\lvert\omega_j\rvert^p)
	 = 0\;  (j \notin I(\omega^*))$, $\omega_j = 0\;  (j \in I(\omega^*))$, and 
	$\lambda \in \Omega_{\epsilon}$. 
\end{proposition}

%%------------------------------------------------------------------ 
%%Proposition 
%\begin{corollary}
%	Let $p < 1$ and $(\omega^*, \lambda^*) \in \Re^n \times \Re^r$ be a local 
%	optimum of \eqref{first-orderoptimality}. Then, $(\omega^*, \lambda^*)$ 
%	together with some pair of vectors
%	$(\zeta^*, \eta^*)\in \Re^n\times \Re^r$ satisfies the SB-KKT conditions 
%	\eqref{SBKKT1}-\eqref{SBKKT6} under an appropriate constraint qualification 
%	concerning the constraints $\frac{\partial G(\omega, 
%	\bar{\lambda})}{\partial \omega_j} + p {\rm 
%	
%sgn}(\omega_j)\lambda_1\lvert\omega_j\rvert^{p-1}\psi'(\lvert\omega_j\rvert^p)
%	 = 0\;  (j \notin I(\omega^*))$, $\omega_j = 0\;  (j \in I(\omega^*))$, and 
%	$\lambda \geq 0$.
%\end{corollary}

\subsection{\alertsection{Bilevel KKT Conditions}}
An immediate consequence of 
Propositions \ref{prop:scaled-lower-level} and 
\ref{prop:sbkkt_is_necessary_for_relaxed} is that when $p\in (0,1)$, a local 
optimum of the one-level problem \eqref{first-orderoptimality} satisfies the 
SB-KKT conditions under appropriate constraint qualifications. However, one 
main drawback of the SB-KKT conditions presented in the preceding section is 
that the process of ``scaling'' enlarges the feasible region of the relaxed 
one-level problem \eqref{first-orderoptimality} when $p=1$. In the following 
definition, we propose an alternative necessary condition which avoids the 
multiplication by $W$ and $W^2$ as defined in Definition~\ref{SBKKTpoint}.

\begin{definition}\label{defn:bkkt}
	We say that $(\omega^*,\lambda^*) \in \Re^n \times \Re^r$ is a bilevel 
	Karush-Kuhn-Tucker point (BKKT point) for problem 
	\eqref{bileveloptimization} if 
	there exists $(\zeta^*,\eta^*)\in 
	\Re^{n}\times 
	\Re^{r}$ such that
		\begin{eqnarray}
		& \nabla_{\tomega} f(\omega^*) + \widetilde{H}(\w^*,{\lambda}^*)  
		\tzeta^*= 0, 
		\label{bkkt1}\\
%		& \nabla_{\tomega} f(\omega^*) + H(\w^*,{\lambda}^*)  \tzeta^* + 
%\nabla^2 _{\tomega\hatomega}G(\omega^*,\bar{\lambda^*}) \hatzeta^*= 0, 
%\label{bkkt1}\\
		& \nabla_{\tomega}G(\omega^*,\bar{\lambda}^*) + p\lambda^*_1 
		\psi'(\lvert\tomega^*\rvert^p) \odot \abs{\tomega^*}^{p-1} \odot 
		\sgn(\tomega^*)= 0, \label{bkkt2}\\
%		& \nabla_{\hatomega}G(\omega^*,\bar{\lambda}^*) = 0 ~\text{if 
%		}\lambda_1^* = 0, \\
		&\hatzeta^* = 0 \ ,\label{bkkt3}\\
		& \left( 
		\psi'(\lvert\tomega^*\rvert^p) \odot \abs{\tomega^*}^{p-1} \odot 
		\sgn(\tomega^*) \right)^T \tzeta^* = \eta_1^*, \label{bkkt4}\\
		& \nabla_{\tomega} \bar{R}(\w^*)^T \tzeta^* - \bareta^* = 0, \ 
		\label{bkkt5} \\
		&  \lambda^*- \epsilon e_1 \geq 0, \quad 
		\eta^* \geq 0, 
		\quad 
		(\lambda^* - \epsilon e_1 )^T \eta^* 
		= 0,
		\label{bkktlast} 
	\end{eqnarray}
where 
{
\begin{eqnarray*}
&\widetilde{H}(\w,\lambda) \coloneqq 
\nabla_{\tomega\tomega}^2G(\w,\barlambda) 
+ \lambda_1p(p-1)\abs{\tilde{W}}^{p-2}\psi'(\abs{\tomega}^p)+p^2\lambda_1 
\abs{\tilde{W}}^{2p-2}\psi''(\abs{\tomega}^p),&\\
&\abs{\tilde{W}} \coloneqq {\rm diag}(\abs{\tomega}),\ \tomega ^* \coloneqq 
(\w^*_i)_{i\notin I(\omega^*)},&\\
&\hatomega ^* \coloneqq (\w^*_i)_{i\in 
	I(\w^*)},\
	\tzeta ^* \coloneqq 
(\zeta^*_i)_{i\notin I(\omega^*)},\ \hatzeta ^* \coloneqq 
(\zeta^*_i)_{i\in I(\w^*)}. &
\end{eqnarray*}}
\end{definition}
We show in the following propositions that the proposed bilevel KKT conditions 
in Definition~\ref{defn:bkkt} are necessary conditions for the one-level 
relaxation \eqref{first-orderoptimality} of the original bilevel problem 
\eqref{bileveloptimization}. In other words, the scaling used in the preceding 
section, as extended from the prior work \cite{OTKW21}, is not needed. 

\begin{proposition}\label{prop:p<1_BKKT}
	Let $p<1$ and $(\omega^*, \lambda^*) \in \Re^n \times \Re^r$ be a local 
	optimum of 
	\eqref{first-orderoptimality}. Then, $(\omega^*, \lambda^*)$ is a BKKT 
	point under an 
	appropriate constraint qualification concerning the constraints 
	$\frac{\partial G(\omega, \bar{\lambda})}{\partial \omega_j} + p {\rm 
		sgn}(\omega_j)\lambda_1\lvert\omega_j\rvert^{p-1}\psi'(\lvert\omega_j\rvert^p)
	= 0\;  (j \notin I(\omega^*))$, $\omega_j = 0\;  (j \in I(\omega^*))$, and 
	$\lambda \in \Omega_{\epsilon}$. 
%	 concerning the constraints 
%	$\frac{\partial G(\omega, \bar{\lambda})}{\partial \omega_j} + p {\rm 
%		
%sgn}(\omega_j)\lambda_1\lvert\omega_j\rvert^{p-1}\psi'(\lvert\omega_j\rvert^p)
%	= 0\;  (j \notin I(\omega^*))$, $\frac{\partial G(\omega, 
%	\bar{\lambda})}{\partial \omega_j} =0$ ($j\in I(\w^*)$) $\omega_j = 0\;  (j 
%	\in I(\omega^*))$, and 
%	$\lambda \geq 0$. 
\end{proposition}
\begin{proof}
	Let $(\w^*,\lambda^*)\in\Re^n\times \Re^r$ be a local minimum of 
	\eqref{first-orderoptimality}, and consider the problem
%		\begin{equation} \label{first-orderoptimality-local}
%		\begin{array}{cc}
%		\min\limits_{\omega,\lambda} & f(\omega) \\
%		{\rm s.t} & 0 = \nabla_{\tomega} G(\w,\barlambda) + p\lambda_1 
%		\psi'(\lvert\tomega\rvert^p) \odot \abs{\tomega}^{p-1} \odot 
%		\sgn(\tomega), \quad \tomega = (\w_i)_{i\notin I(\w^*)}\\
%		&  0 \in \nabla_{\hatomega} G(\w,\barlambda) + p\lambda_1 
%		\psi'(0) \odot  \partial (\abs{\hatomega}^p)\big\rvert_{\hatomega=0} , 
%		\quad \hatomega = (\w_i)_{i\in I(\w^*)}\\
%		& \hatomega = 0 \\ 
%		& \lambda \geq 0.
%		\end{array}
%		\end{equation} 
%	Suppose first that $\lambda_1^* > 0$, and consider the problem
			\begin{equation} \label{first-orderoptimality-local-lambda1>0}
			\begin{array}{cc}
			\min\limits_{\omega,\lambda} & f(\omega) \\
			{\rm s.t} &  \nabla_{\tomega} G(\w,\barlambda) + p\lambda_1 
			\psi'(\lvert\tomega\rvert^p) \odot \abs{\tomega}^{p-1} \odot 
			\sgn(\tomega) = 0 , \\
			& \hatomega = 0 , \\ 
			& \lambda \in \Omega_{\epsilon},
			\end{array}
			\end{equation} 
	where $\tomega = (\w_i)_{i\notin I(\w^*)}$ and $\hatomega = (\w_i)_{i\in I(\w^*)}$. We claim that $(\w^*,\lambda^*)$ is a local minimum of 
	\eqref{first-orderoptimality-local-lambda1>0}. It is clear that 
	$(\w^*,\lambda^*)$ is feasible to 
	\eqref{first-orderoptimality-local-lambda1>0}. Moreover, if $(\w,\lambda)$ 
	is a feasible point of \eqref{first-orderoptimality-local-lambda1>0}, 
	it follows from the first equality constraint that $\tomega\neq 0$ since 
	$p\in (0,1)$ and $\abs{\tomega}^{p-1} \in \Re^{n-\abs{I(\omega^*)}}$. 
	Then with the first equality constraint in 
	\eqref{first-orderoptimality-local-lambda1>0} together with the fact that 
	$\partial_{\hatomega}  \left( G(\w,\barlambda) + 
	\lambda_1 R_1(\w)\right)  = 
	\Re^{\abs{I(\w^*)}}$, it 
	immediately follows that
	$0 \in \partial_{\omega}(G(\omega, \bar{\lambda}) +  \lambda_1 
	R_1(\omega))$, and therefore $(\w,\lambda)$ belongs to the feasible region 
	of \eqref{first-orderoptimality}. As $(\w^*,\lambda^*)$ is a local minimum 
	of \eqref{first-orderoptimality} and the feasible region of 
	\eqref{first-orderoptimality-local-lambda1>0} is contained in that of 
	\eqref{first-orderoptimality}, we indeed have that $(\w^*,\lambda^*)$ is a 
	local minimum of \eqref{first-orderoptimality-local-lambda1>0}. Hence, 
	there exist 
	Lagrange 
	multipliers 
	$(\hat{\zeta}^{(1)},\hat{\zeta}^{(2)},\hat{\eta})\in\Re^{n-\abs{I(\w^*)}}\times
	 \Re^{\abs{I(\w^*)}} \times \Re^r$ such that 
		\begin{eqnarray}
		\begin{bmatrix}
		\nabla_{\tomega} f(\w^*) \\
		\nabla_{\hatomega} f(\w^*) \\
		0 \\
		0		
		\end{bmatrix} + 
		\left[\begin{array}{cc}
	 	 \widetilde{H}(\w^*,\barlambda^*) & 
		0 \\
		\nabla_{\hatomega\tomega}^2 G(\w^*,\barlambda^*) & I \\
		p\psi'(\lvert\tomega ^*\rvert^p) \odot \abs{\tomega^*}^{p-1} \odot 
		\sgn(\tomega^*) & 0 \\
		\nabla_{\tomega}\bar{R}(\w^*)^T & 0
		\end{array}\right ] \begin{bmatrix}
		\hat{\zeta}^{(1)} \\ 
		\hat{\zeta}^{(2)}
		\end{bmatrix} - \begin{bmatrix}
		0 \\ 
		0 \\
		\hat{\eta}_1 \\ 
		\bar{\hat{\eta}}
		\end{bmatrix} = 0 , \label{kkt1_lambda1>0}\\
		\nabla_{\tomega} G(\w^*,\barlambda)^* + p\lambda_1 ^*
		\psi'(\lvert\tomega^*\rvert^p) \odot \abs{\tomega^*}^{p-1} \odot 
		\sgn(\tomega^*) = 0 ,\\
		\lambda^*- \epsilon e_1 \geq 0, \quad 
		\eta^* \geq 0, 
		\quad 
		(\lambda^* - \epsilon e_1)^T \eta^* 
		= 0	, 
		\label{kktlast_lambda1>0}			
		\end{eqnarray}		
		where $\bar{\hat{\eta}}=(\hat{\eta}_2,\hat{\eta}_3 
		,\dots,\hat{\eta}_r)^T$. 
%		Since $\lambda_1^*>0$, it follows from 
%		\eqref{kktlast_lambda1>0} that 
%		$\hat{\eta}_1^* =0$ and so from 
%		\eqref{kkt1_lambda1>0}, we obtain $p\psi'(\lvert\tomega ^*\rvert^p) 
%		\odot 
%		\abs{\tomega^*}^{p-1} \odot 
%		\sgn(\tomega^*) = 0$. 
		Taking $\eta^* = \hat{\eta}$ and setting $\zeta^* \in \Re^n$ such that 
		$\zeta_i^* 
		= 
		\hat{\zeta}^{(1)}_i$ for $i\notin I(\w^*)$ and $\zeta_i^* = 0$ for 
		$i\in 
		I(\w^*)$, one can easily check 
		that all the conditions 
		\eqref{bkkt1}-\eqref{bkktlast} are satisfied.
		
%		Consider now the case that $\lambda_1^* = 0$, and consider the problem 
%		\begin{equation} \label{first-orderoptimality-local-lambda1=0}
%		\begin{array}{cc}
%		\min\limits_{\omega,\lambda} & f(\omega) \\
%		{\rm s.t} &  \nabla_{\w} G(\w,\barlambda) = 0 , \\
%		& \hatomega = 0 , \quad \hatomega = (\w_i)_{i\in I(\w^*)}\\ 
%		& \lambda_1 =0, \quad \bar{\lambda} \geq 0.
%		\end{array}
%		\end{equation}
%		Similar to the previous case, it can be shown that $(\w^*,\lambda^*)$ 
%		is a local minimum of \eqref{first-orderoptimality-local-lambda1=0} by 
%		arguing that it is feasible to 
%		\eqref{first-orderoptimality-local-lambda1=0} and the feasible region 
%		of \eqref{first-orderoptimality} contains the feasible region of 
%		\eqref{first-orderoptimality-local-lambda1=0}. Hence, we can obtain 
%		Lagrange multipliers $(\zeta^*,\eta^*)\in \Re^n \times \Re^r$ that 
%		satisfies the KKT conditions of 
%		\eqref{first-orderoptimality-local-lambda1=0}. It is straightforward to 
%		show that $(\w^*,\lambda^*)$ together with $(\zeta^*,\bar{\eta}^*)$ 
%		where $\bar{\eta}^* = (\eta_2^*,\dots,\eta_r^*)$ satisfies the BKKT 
%		conditions \eqref{bkkt1}-\eqref{bkktlast}. 
\end{proof}

For the case $p=1$, we obtain a similar result provided that the feasibility 
condition
	\begin{equation}\label{eq:interiortosubdifferential}
		\|\nabla_{\hatomega}G(\w^*,\barlambda^*) \|_{\infty} < \lambda_1^* 
		\psi'(0)
	\end{equation}
holds, where $\hatomega = 
(\w_i)_{i\in I(\w^*)}$. We note that this is 
equivalent to saying that $-\frac{1}{\lambda_1^*} \frac{\partial G}{\partial 
\w_i}(\w^*,\barlambda^*)$ belongs to the interior of the subdifferential set 
$\partial \psi (\abs{t}) 
\vert_{t=0}$ for all $i\in I(\w^*)$. This condition is used in the convergence 
analysis of the smoothing algorithm in \cite{OTKW21}, but its connection with 
the necessary conditions for solutions of \eqref{first-orderoptimality} was not 
explored. This precise connection is revealed in the following proposition. 

\begin{proposition}
	Let $p=1$ and $(\omega^*, \lambda^*) \in \Re^n \times \Re^r$ be a local 
	optimum of 
	\eqref{first-orderoptimality} that satisfies 
	\eqref{eq:interiortosubdifferential}. Then, $(\omega^*, \lambda^*)$ is a 
	BKKT 
	point under an 
	appropriate constraint qualification concerning the constraints 
	$\frac{\partial G(\omega, \bar{\lambda})}{\partial \omega_j} + p {\rm 
		sgn}(\omega_j)\lambda_1\lvert\omega_j\rvert^{p-1}\psi'(\lvert\omega_j\rvert^p)
	= 0\;  (j \notin I(\omega^*))$, $\omega_j = 0\;  (j \in I(\omega^*))$, and 
	$\lambda \in \Omega_{\epsilon}$. 
	%	 concerning the constraints 
	%	$\frac{\partial G(\omega, \bar{\lambda})}{\partial \omega_j} + p {\rm 
		%		
		%sgn}(\omega_j)\lambda_1\lvert\omega_j\rvert^{p-1}\psi'(\lvert\omega_j\rvert^p)
	%	= 0\;  (j \notin I(\omega^*))$, $\frac{\partial G(\omega, 
		%	\bar{\lambda})}{\partial \omega_j} =0$ ($j\in I(\w^*)$) $\omega_j = 
		%0\;  (j 
	%	\in I(\omega^*))$, and 
	%	$\lambda \geq 0$. 
\end{proposition}
\begin{proof}
	We consider a problem similar to 
	\eqref{first-orderoptimality-local-lambda1>0} but with 
	\eqref{eq:interiortosubdifferential} as an added inequality constraint, 
	that is, 
	\begin{equation} \label{first-orderoptimality-local-lambda1-p=1}
		\begin{array}{cc}
			\min\limits_{\omega,\lambda} & f(\omega) \\
			{\rm s.t} &  \nabla_{\tomega} G(\w,\barlambda) + \lambda_1 
			\psi'(\lvert\tomega\rvert) \odot 
			\sgn(\tomega) = 0 , \\
			& \hatomega = 0 , \\ 
			& \lambda \in \Omega_{\epsilon} \\
			& \|\nabla_{\hatomega}G(\w,\barlambda) \|_{\infty} < \lambda_1 
			\psi'(0) ,
		\end{array}
	\end{equation} 
where $\tomega = (\w_i)_{i\notin I(\w^*)}$ and $\hatomega = (\w_i)_{i\in I(\w^*)}$. Note that $\|\nabla_{\hatomega}G(\w,\barlambda) \|_{\infty} < \lambda_1 
\psi'(0)$ is a non-binding inequality 
constraint of \eqref{first-orderoptimality-local-lambda1-p=1}. Hence, following 
the proof of Proposition~\ref{prop:p<1_BKKT}, it 
suffices to show 
that $(\w^*,\lambda^*)$ is feasible to 
\eqref{first-orderoptimality-local-lambda1-p=1} and that the feasible region of 
\eqref{first-orderoptimality-local-lambda1-p=1} is contained in that of 
\eqref{first-orderoptimality}. The former is clear due to our hypothesis. To 
show the inclusion of the feasible regions, let $(\w,\lambda)$ be a feasible 
point of \eqref{first-orderoptimality-local-lambda1-p=1}. If $i\in I(\w^*)$, 
then $\w_i = 0$, which together with \eqref{eq:interiortosubdifferential} 
implies that $0\in \partial_{\omega_i}(G(\w,\barlambda)+\lambda_1 R_1(\w))$. If 
$i\notin I(\w^*)$ and $\w_i \neq 0$, it is clear that $0\in 
\partial_{\omega_i}(G(\w,\barlambda)+\lambda_1 R_1(\w))$ from the first 
equality constraint in \eqref{first-orderoptimality-local-lambda1-p=1}. On the 
other hand, if $i\notin I(\w^*)$ but $\w_i=0$, we also have from the first 
equality 
constraint in \eqref{first-orderoptimality-local-lambda1-p=1} that $\partial 
G(w,\barlambda)/\partial \w_ i =0$. Since $0\in \partial \psi 
(\abs{t})\vert_{t=0}$, we again have $0\in 
\partial_{\omega_i}(G(\w,\barlambda)+\lambda_1 R_1(\w))$. This completes the 
proof. 
\end{proof}

\begin{remark}\label{remark:p=1} We make some comments about the case $p=1$.
	\begin{description}
		\item[\rm (a)] For this case, we always need to assume that a candidate 
		bilevel KKT 
		point
		$(\w^*,\lambda^*)$ satisfies inequality 
		\eqref{eq:interiortosubdifferential}. This will also be the standing 
		assumption for our subsequent analysis when dealing with the case of 
		$p=1$, as we shall see in the next section.
		\item[\rm (b)] Note that if $g$ and $R_j$, 
		$j=2,\dots, r$ are convex functions, then we obtain a 
		stronger result that \emph{the bilevel \tblue{KKT} conditions are necessary 
		conditions 
			for the original bilevel problem \eqref{bileveloptimization}} under 
		appropriate constraint 
		qualifications, rather than just necessary conditions for the relaxed 
		problem \eqref{first-orderoptimality}, which is the situation when 
		$p\in 
		(0,1)$ (even if the functions $g$ and $R_j$ are all convex). Hence, in 
		this case, bilevel KKT points are indeed candidate solutions to the 
		bilevel problem \eqref{bileveloptimization}.  
	\end{description}
	  
\end{remark}
 \section{Proposed algorithm \alert{and its convergence}}
 
In this section, we describe our smoothing algorithm for \eqref{bileveloptimization} and present our convergence results.
 
 \subsection{Smoothing approach and the algorithm}
 One main source of difficulty in solving the bilevel program \eqref{bileveloptimization} is the nonsmooth, nonconvex and possibly non-Lipschitz component $R_1(\omega)=\sum _{i=1}^n \psi (\lvert\omega_i\rvert^p)$, where $p\in (0,1]$.
     To overcome this difficulty, we apply the smoothing technique to $R_1$ with the smoothing function $\phi$ defined in the previous section, yielding the following smooth function:

     %% Our goal is to approximate our problem \eqref{bileveloptimization} by a one-level smooth optimization problem.
     %% By virtue of Assumption~(A) and our differentiability assumptions on $R_i$ ($i=2,\dots , r$) and $g$, our objective can be accomplished using the KKT conditions for the lower-level problem provided that the chosen smooth approximation of $R_1$ is twice continuously differentiable. It is not difficult to show that such a smoothing function of $R_1$ can be obtained by choosing a smooth approximation $\phi(\mu,x)$ of $\lvert x\rvert$ that is strictly positive on $\Re$. In this case, the function
  \begin{equation}\label{Smooth-R1}
 \varphi_{\mu}(\w) \coloneqq \sum_{j=1}^{n}\psi\left([\phi(\mu, \w_j)]^p\right).
 \end{equation}
  Then, as in \cite{OTKW21}, we consider problem\,\eqref{bileveloptimization} with $\varphi_{\mu}$ in place of $R_1$, and further replace the obtained smoothed lower-level problem with its first-order condition. Hence, the following problem is obtained: 
%  \noindent is a smoothing function of $R_1$ that is twice continuously differentiable.  Consequently, \tred{the following one-level problem
%approximates problem \eqref{bileveloptimization}:}
\begin{equation} \label{one-levelproblem}
\begin{array}{cc}
\min\limits_{\omega,\lambda} & f(\omega) \\
{\rm s.t} & \nabla_{\omega}G(\omega, \bar{\lambda}) + \lambda_1\nabla\varphi_{\mu}(\omega) = 0 \\
& \alert{\lambda \in \Omega_{\epsilon}}.
\end{array}
\end{equation}

      Next, let us suppose that $\varphi_{\mu}$ is twice continuously differentiable from this moment and recall Assumption~(A) together with our differentiability assumptions on $R_i$ ($i=2,\dots , r$) and $g$. These properties enable us to consider the KKT conditions.\footnote{Without the $\mathcal{C}^2$ property of $\varphi_{\mu}$, the KKT conditions cannot be taken because the constraint function $\nabla_{\omega}G(\omega, \bar{\lambda}) + \lambda_1\nabla\varphi_{\mu}(\omega)$ is not necessarily smooth.}
    %% Moreover, since $\varphi_{\mu }$ is twice continuously differentiable,
    By virtue of this fact, we can find candidate solutions to \eqref{one-levelproblem} by looking at its KKT points.

    In fact, it is sufficient to obtain approximate KKT points:
Given a parameter $\hat{\varepsilon}>0$, we define an 
$\hat{\varepsilon}$-approximate KKT point for problem \eqref{one-levelproblem} 
as follows: We say that $\{(\omega, \lambda, \zeta, \eta)\} \subseteq \Re^n 
\times \Re^r \times \Re^n \times \Re^r$ is an $\hat{\varepsilon}$-approximate 
KKT point for \eqref{one-levelproblem} if there exists a vector 
$\hat{\varepsilon} = 
(\varepsilon_1,\varepsilon_2,\varepsilon_3,\varepsilon_4,\varepsilon_5) \in 
\Re^n \times \Re \times \Re^{r-1} \times \Re^n \times \Re$ such that
\begin{eqnarray}
& \nabla f(\omega) + (\nabla^2_{\omega\omega}G(\omega, \bar{\lambda}) + \lambda_1\nabla^2\varphi_{\mu}(\omega))\zeta = \varepsilon_1, \label{KKT1}\\
& \nabla\varphi_{\mu}(\omega)^T \zeta - \eta_1 = \varepsilon_2, \label{KKT2}\\
& \nabla R_j(\omega)^T \zeta - \eta_j = (\varepsilon_3)_{\alert{j-1}} \ (j=2,3, \dots, r), \label{KKT3}\\
& \nabla_{\omega}G(\omega, \bar{\lambda}) + \lambda_1\nabla\varphi_{\mu}(\omega) = \varepsilon_4, \label{KKT4}\\
&\lambda - \epsilon e_1 \geq 0, \quad \eta \geq 0 , \quad  (\lambda- \epsilon 
e_1)^T \eta 
= \varepsilon_5, \label{KKT5} 
\end{eqnarray}
and
\[
\|\alert{(\varepsilon_1^T,\varepsilon_2^T,\varepsilon_3^T,\varepsilon_4^T,\varepsilon_5^T)^T}\| \leq \hat{\varepsilon},
\]
where $\nabla^2_{\omega\omega}G(\omega, \bar{\lambda})$ is the Hessian of $G$ with respect to $\omega$. Note that when $\hat{\varepsilon} = 0$, an $\hat{\varepsilon}$-approximate KKT point is identical to a KKT point. 

\medskip 

Now, by iteratively computing an $\hat{\varepsilon}$-approximate KKT point 
while decreasing $\hat{\varepsilon}$ and the smoothing parameter $\mu$, we 
obtain the smoothing algorithm presented in Algorithm~\ref{algorithm}.

\medskip 

\begin{algorithm} 
\caption{(A Smoothing Method for Nonsmooth Bilevel Optimization)}\label{algorithm}
%\begin{algorithmic}[1]
	\begin{description}
		\item [Step 0] Choose $\mu_0\alert{>} 0$, $\beta_1,\beta_2 \in (0,1)$ 
		and $\hat{\varepsilon}_0\geq 0$. Set $k\coloneqq0$.
		
		\item [Step 1]  Find an $\hat{\varepsilon}_k$-approximate KKT point $\{(\omega^{k+1}, \lambda^{k+1}, \zeta^{k+1}, \eta^{k+1})\}$ for problem \eqref{one-levelproblem} with $\mu=\mu_k$.
		
		\item [Step 2] Set $\mu_{k+1}=\beta_1\mu_k$,  
		$\hat{\varepsilon}_{k+1}=\beta_2\hat{\varepsilon}_k$ and 
		$k\coloneqq k+1$.
	\end{description}
%\end{algorithmic}
\end{algorithm}
  Algorithm\,\ref{algorithm} is quite similar to the one proposed by Okuno~et al\,\cite{OTKW21}. However, whereas Okuno~et al supposed to employ only $\sum_{i=1}^n(\w_i^2+\mu^2)^{\frac{p}{2}}$ as the smoothing function $\varphi_{\mu}$, Algorithm\,\ref{algorithm} enjoys much more freedom in choices.

{In our convergence analysis, we assume that at every iteration, an 
$\hat{\varepsilon}_k$-approximate KKT 
point can always be computed.} In order to establish the global convergence of 
Algorithm\,\ref{algorithm}, we will require some more properties of the density 
function $\rho$ used for constructing the smoothing function $\phi$, which also 
accomplish the $\mathcal C^2$-property of $\varphi_{\mu}$ supposed above.

\medskip

%------------------------------------------------------------------------------
  \subsection{Convergence analysis}

\subsubsection{{Assumptions on density 
function}}\label{smoothingapproach}
%% We note that the above discussion was under the premise that $\phi (\mu,x)$ 
%%is a strictly positive function, which is one among other important 
%%properties 
%%that we require of $\phi(\mu,x)$. Meanwhile,
We have mentioned in the Introduction that the 
setting of all our analysis is 
based on density functions. That is, we wish to prove all our convergence 
results by solely looking at density functions used to induce the smoothing 
functions. To this end, we must be able to identify necessary properties of a 
given density function so that Algorithm~\ref{algorithm} converges to a 
candidate solution of the main problem \eqref{bileveloptimization}. Indeed, one
novelty of our work is precisely the identification of these required 
properties and its application to the convergence analysis. 

\medskip
We summarize necessary assumptions on the density function $\rho$ that we will 
use in the next subsection.

\medskip 

\medskip 

\noindent \textbf{Assumption \alert{(B)}.} Let $\rho: \Re \to \Re_+$ be a 
density 
function. Then, the following properties hold:
\begin{description}[font=\normalfont]
	\item[(B1)] $\rho$ is symmetric, i.e. $\rho(x)=\rho(-x)$ for all $x\in \Re$.
	\item[(B2)] $\rho$ is continuous and nonincreasing on $[0,\infty)$.
	\item[(B3)] There exist positive constants $c,r>0$ such that 
	\[
	2\int\limits_{0}^{S}\rho(x) ~dx  \geq 1 - \frac{c}{S^r + c} ~~ \mbox{for 
	all} ~ S\geq 0.
	\]
	\item [(B4)] If $p=1$, we have $\rho(x)>0$ for all $x\in \Re$.
\end{description}

Some remarks are in order: First,
although Assumption~(B1) was already supposed in Section 
\ref{subsec:smoothingfunctions}, we have restated it for later use. Under this 
assumption, note that the smoothing function $\phi(\mu,x)$ is strictly positive 
for all $\mu>0$ and $x\in \Re$. Indeed, we already have from Proposition 
\ref{properties-smoothabsolutevalue}(b) that $\phi(\mu,x)>0$ for $x\neq 0$. On 
the other hand, since $\rho$ is a symmetric density function by Assumption 
(B1), then $\rho$ is not identical to the zero function on the interval 
$[0,\infty)$. Consequently, $\int_{0}^{+\infty}t\rho(t)dt >0$, which together 
with \eqref{smoothplus} and \eqref{smoothabsolutevalue} yields $\phi(\mu,0)>0$. 
Moreover, we can easily calculate the first and second derivatives of the 
induced $\phi(\mu,x)$ as
\begin{equation}\label{phi_firstder}
\phi'(\mu,x) = 2 \mathrm{sgn}(x)\int_{0}^{\lvert 
x\rvert}\frac{1}{\mu}\rho\left(\frac{t}{\mu}\right) \; dt = 2 
\mathrm{sgn}(x)\int_{0}^{\frac{\lvert x\rvert}{\mu}}\rho(t)\; dt,
\end{equation}
and
\begin{equation}\label{phi_secondder}
\phi''(\mu,x) = \frac{2}{\mu}\rho\left(\frac{x}{\mu}\right),	
\end{equation}
respectively. From equation \eqref{phi_secondder} and strict positivity of 
$\phi(\mu,x)$, we see that $\phi(\mu, \cdot)$ is twice continuously 
differentiable approximation of $\lvert x\rvert$ by the continuity assumption 
in (B2). 

Assumptions~(B1) and (B2) will also have other important roles in the proofs of 
our main result. \alert{Among them are formulas for the limits of $\{ 
	(\nabla\varphi_{\mu_{k-1}}(\omega^k))_j\}$ and $\{ 
	(\nabla^2\varphi_{\mu_{k-1}}(\omega^k))_{jj}\}$ when $j\notin I(\omega^*)$, 
	where $\varphi_{\mu }$ is the smooth function given by \eqref{Smooth-R1}.
	
%	First, it can be easily calculated that
\tblue{First, with easy calculation,} 
	the components of 
	$\nabla\varphi_{\mu}(\omega) \in \Re^n$ are given by 
	\begin{equation}\label{Gradient-varphi}
	(\nabla\varphi_{\mu}(\omega))_j = p \psi'\left([\phi(\mu, 
	\omega_j)]^p\right) \phi'(\mu, \omega_j)[\phi(\mu, \omega_j)]^{p-1},
	\end{equation}
	while $\nabla^2\varphi_{\mu}(\omega)\in \Re^{n\times n}$ is a diagonal 
	matrix 
	whose diagonal entries are given by 
	\begin{align}
	(\nabla^2\varphi_{\mu}(\omega))_{jj} &= p^2\psi''\left([\phi(\mu, 
	\omega_j)]^p\right)[\phi'(\mu, \omega_j)[\phi(\mu, \omega_j)]^{p-1}]^2 
	\notag{}\\
	&+ p(p-1)\psi'\left([\phi(\mu, \omega_j)]^p\right) [\phi'(\mu, 
	\omega_j)]^2[\phi(\mu, \omega_j)]^{p-2} \notag{}\\ 
	&+ p\psi'\left([\phi(\mu, \omega_j)]^p\right)  [\phi(\mu, \omega_j)]^{p-1} 
	\phi''(\mu, \omega_j),  \label{Hessian-matrix}
	\end{align}
	for $j=1,\dots, n$. 	
%	Note that
%	the smoothing function $\phi(\mu,x)$ is strictly positive under Assumption~(B1) 
\tblue{Since $\phi(\mu,x)$ is strictly positive under Assumption~(B1),} the factor $[\phi(\mu,\omega_j)]^{p-1}$ that appears in 
	\eqref{Gradient-varphi} and \eqref{Hessian-matrix} is real-valued. In 
	addition, 
	it is clear from \eqref{Hessian-matrix} that the components of the Hessian 
	of 
	$\varphi_{\mu }$ are continuous by Assumption~(A), equation 
	\eqref{phi_secondder}, and Assumptions~(B1)-(B2).}

\alert{We now list some important formulas but will be useful in our subsequent 
	analysis. }

\alert{\begin{lemma}\label{lemma:limit_gradient_hessian}
		Suppose that Assumptions~(B1)-(B2) hold. Let $\{ (\omega ^k, 
		\mu_{k-1})\} 
		\subseteq \Re^n \times \Re_{++} $ be an arbitrary sequence converging 
		to 
		$(\omega ^*, 0)$. Then for $j\notin I(\omega^*)$,
		\begin{align}
		\lim_{k \to \infty}(\nabla\varphi_{\mu_{k-1}}(\omega^k))_j & = 
		p\mathrm{sgn}(\omega^*_j) 
		\lvert\omega^*_j\rvert^{p-1}\psi'\left(\lvert\omega^*_j\rvert^{p}\right),
		\label{prop41.1}\\
		\lim_{k\to\infty}(\nabla^2\varphi_{\mu_{k-1}}(\omega^k))_{jj} &= 
		p^2\psi''\left(\lvert\omega^*_j\rvert^{p}\right)\lvert\omega^*_j\rvert^{2p-2}
		+ 
		p(p-1)\psi'\left(\lvert\omega^*_j\rvert^{p}\right)\lvert\omega^*_j\rvert^{p-2}.
		\label{prop41.2}
		\end{align}
	\end{lemma}
	
	\medskip 
	
	\begin{proof}
		Let $j\notin I(\omega^*)$. We have from 
		\eqref{bounded-sequence-smoothabsolute} that
		\[
		\lim_{k\to \infty}\phi'(\mu_{k-1}, \omega^k_j) = 
		\mathrm{sgn}(\omega^*_j).
		\]
		Since $\phi(\mu,x)$ is a smoothing function of $\lvert x \rvert$ that 
		is 
		strictly positive by Assumption~(B1), we have $[\phi(\mu_{k-1}, 
		\omega^k_j)]^{p-1} \to  \lvert\omega^*_j\rvert^{p-1}$ as $k\to \infty$. 
		Moreover, since $\psi$ is $\mathcal C^2$ by Assumption~(A), then we 
		easily 
		obtain \eqref{prop41.1} by letting $k\to\infty$ in equation 
		\eqref{Gradient-varphi} with $\w = \w^k$ and $\mu = \mu_{k-1}$. 
		
		For equation \eqref{prop41.2}, we first show that $\lim_{k\to\infty} 
		\phi''(\mu_{k-1}, \omega^k_j) =0$, which by \eqref{phi_secondder} is 
		equivalent to showing that
		\begin{equation}\label{claim,limit=0}
		\lim_{k\to\infty} 
		\frac{1}{\mu_{k-1}}\rho\left(\frac{\omega^k_j}{\mu_{k-1}}\right) = 
		\lim_{k\to\infty} 
		\frac{1}{\mu_{k-1}}\rho\left(\frac{\lvert\omega^k_j\rvert}{\mu_{k-1}}\right)
		=0.
		\end{equation}
		Since $j\notin I(\omega^*)$, then 
		$\lvert\omega^*_j\rvert/2<\lvert\omega^k_j\rvert$ for all $k$ 
		sufficiently 
		large. Thus,
		\[0\leq \frac{1}{\mu_{k-1}} \rho \left( \frac{\omega _j^k}{\mu_{k-1}} 
		\right) \leq \frac{1}{\mu_{k-1}}\rho \left( \frac{\omega 
			_j^*}{2\mu_{k-1}} 
		\right) \to \delta (\omega_j^*\alert{/2})=0 \qquad 
		\text{as}~k\to\infty,\]
		by invoking Assumptions~(B1)-(B2) and the definition of the Dirac delta 
		function. This proves \eqref{claim,limit=0}. Finally, taking the limit 
		in 
		\eqref{Hessian-matrix} when $k\to \infty$, we obtain
		\begin{align*}
		&\lim_{k\to\infty}(\nabla^2\varphi_{\mu_{k-1}}(\omega^k))_{jj}\\
		&= 
		p^2\psi''\left(\lvert\omega^*_j\rvert^{p}\right)[\mbox{sgn}(\omega^*_j)\lvert\omega^*_j\rvert^{p-1}]^2
		+ 
		p(p-1)\psi'\left(\lvert\omega^*_j\rvert^{p}\right)[\mbox{sgn}(\omega^*_j)]^2\lvert\omega^*_j\rvert^{p-2}\\
		&= 
		p^2\psi''\left(\lvert\omega^*_j\rvert^{p}\right)\lvert\omega^*_j\rvert^{2p-2}
		+ 
		p(p-1)\psi'\left(\lvert\omega^*_j\rvert^{p}\right)\lvert\omega^*_j\rvert^{p-2}.
		\end{align*}
		This completes the proof of the lemma. 
\end{proof}}

On the other hand, \alert{the other technical assumption on 
	$\rho$, namely 
	(B3), 
	is} important in our subsequent analysis. Without knowing 
definitively the formula for $\phi (\mu,x)$, the analysis becomes extremely 
difficult. In particular, it is challenging to understand the convergence 
behavior of the sequence $\left\lbrace S_j^k \right\rbrace$, where 
\begin{equation}\label{Sjk}
S_j^k \coloneqq \frac{\lvert\omega_j ^k\rvert}{\mu_{k -1}}, 	
\end{equation}
$\omega^k = (\omega_1^k, \dots, \omega _n^k )$ and $\mu_{k-1}$ are generated 
from Algorithm~\ref{algorithm} when $\omega_j^k \to 0$ as $k\to\infty$ for some 
$j\in \{1,\dots, n\}$. Nonetheless, this problem can be alleviated thanks to 
the simple Assumption~(B3). 
Finally, Assumption~(B4) will later be important in proving the 
unboundedness of the sequence $\left\lbrace 
\lvert\nabla^2\varphi_{\mu_{k-1}}(\omega^k) )_{jj}\rvert\right\rbrace$ when 
$p=1$ and $\omega_j^k \to 0$ as $k\to \infty$. Interestingly, we shall see 
shortly that (B4) is not needed for the case $p\in (0,1)$. 

\alert{In the forthcoming discussions}, we will 
see in great detail how these assumptions on 
$\rho$ will play a central role in establishing the main convergence result.   
In \alert{
	Appendix~\ref{app:examples_smoothing}}, we will provide some specific examples 
	of 
density functions satisfying Assumption~(B).

%---------------------------------------------------------------------------
\subsubsection{Subsequential convergence}

We now prove our main result that accumulation points of the sequence generated 
by Algorithm~\ref{algorithm} are in fact \alert{bilevel KKT points (see 
Definition~\ref{defn:bkkt}), which in turn are candidate solutions for 
\eqref{first-orderoptimality} when $p<1$, and candidate solutions for the original bilevel problem \eqref{bileveloptimization}}. \alert{As mentioned in Remark~\ref{remark:p=1}, 
we will 
assume that for any given accumulation point of such sequence generated by 
Algorithm~\ref{algorithm}, the inequality \eqref{eq:interiortosubdifferential} 
holds when $p=1$}. 

%----------------------------------------------------------------------------------------------------------
% Theorem 3.2
\begin{theorem}\label{thm:convergence-analysis}
	
	\alert{ Let $p\in (0,1]$ and assume that $(\omega ^*, 
	\lambda^*,\zeta^*,\eta^*)$ is any accumulation 
	point of 
		a 
		sequence $\{(\omega^k, \lambda^k,\zeta^k,\eta^k)\}$ generated by 
		Algorithm 
		\ref{algorithm}. Then 
		$(\omega^*,\lambda^*)$ is a 
		bilevel KKT point for the original problem 
		\eqref{bileveloptimization} provided that Assumptions~(A) and (B) 
		hold}. 
\end{theorem}

\alert{To prove Theorem 
	\ref{thm:convergence-analysis}, we show that 
	$(\omega^*,\lambda^*,\zeta^*,{\eta}^*)$ satisfies equations 
	\eqref{bkkt1}-\eqref{bkktlast}. 
	%	The main difficulty lies in showing that 
	%	$\hatzeta ^ * = 0 $ when $\lambda_1^* > 0$ in order to prove 
	%\eqref{bkkt3}. 
	%	To this end, we prove a series of lemmas.}
	To this end, we prove a series of lemmas, and in particular, we do the 
	following: }

\begin{description}[font=\normalfont]
	\item[(i)] Prove that $\{ S_j^k\}_{k\in K}$ is bounded, where $S_j^k$ 
	is given by \eqref{Sjk}, $j\in I(\omega ^*)$, $K\subset \{ 
	1,2,\dots,\}$ such that $(\omega ^k,\lambda^k) \to (\omega^*, 
	\lambda^*)$ as $k\in K \to \infty$ and $\{ (\omega ^k, \lambda^k)\}$ is 
	generated by Algorithm~\ref{algorithm}; 
	%		Such an index set $K$ exists by Assumption~(B2); 
	\item[(ii)] \alert{Using the boundedness of $\{ S_j^k\}_{k\in K}$, we 
		compute} the limit of the sequence $\left\lbrace 
	(\nabla^2\varphi_{\mu_{k-1}}(\omega^k) )_{jj}\right\rbrace_{k\in K}$, 
	where $j\in I(\omega ^*)$ and the index set $K$ is as described in (i); 
	%		\item[(iii)] Compute the limit of the sequences $\left\lbrace 
	%(\nabla\varphi_{\mu_{k-1}}(\omega^k) )_{j}\right\rbrace$ and $\left\lbrace 
	%(\nabla^2\varphi_{\mu_{k-1}}(\omega^k) )_{jj}\right\rbrace$, where 
	%$j\notin 
	%I(\omega ^*)$; 
	%		\item[(iv)] Prove that the sequence $\{(\zeta ^k , \eta ^k) \}$ 
	%generated by Algorithm~\ref{algorithm} is bounded, and thus, accumulation 
	%points of the sequence of iterates exist;
	%		\item[(iii)] Prove that any accumulation point $(\omega^*, 
	%\lambda^*, 
	%		\zeta^*, \eta^*)$ of the sequence $\{(\omega^k, \lambda^k, \zeta^k, 
	%		\eta^k)\}$ generated by Algorithm~\ref{algorithm} satisfies 
	%equations 
	%		\eqref{SBKKT3} and \eqref{SBKKT4}; and finally,
	%		\item[(vi)] Compute the limits of the sequences $\left\lbrace W_k 
	%\nabla\varphi_{\mu_{k-1}}(\omega^k) \right\rbrace _{k\in K}$ and 
	%$\left\lbrace 
	%W_k^2\nabla^2\varphi_{\mu_{k-1}}(\omega^k) \right\rbrace _{k\in K}$ for an 
	%arbitrary sequence $\{(\omega ^k, \mu_{k-1} )\}$ converging to $(\omega 
	%^*, 
	%0)$, where $W_k \coloneqq {\rm diag}(\omega ^k)$.
	\item[\alert{(iii)}] \alert{Using (ii) and Lemma 
		\ref{lemma:limit_gradient_hessian}, we show that $\hatzeta^* = 0$ and 
		equation \eqref{bkkt4} holds.}
\end{description}

The above objectives are formally stated and proved, respectively, in Lemma 
\ref{lemma:boundedratio} to Lemma~\ref{lemma:lagrange_multipliers}. We will 
prove these results without 
knowledge of the exact formula for the smoothing function $\phi (\mu,x)$ used 
to construct $\varphi_{\mu }$ given by \eqref{Smooth-R1}, that is, only using 
Assumption~(B) on the density function.	

\medskip 
\alert{To facilitate our subsequent analysis, we note here that a sequence 
	$\{(\omega^k,\lambda^k,\zeta^k,\eta^k)\}$ generated by Algorithm 
	\ref{algorithm} satisfies 
	\begin{eqnarray}
	& \nabla f(\omega^k) + (\nabla^2_{\omega\omega}G(\omega^k, \bar{\lambda}^k) 
	+ 
	\lambda_1^k\nabla^2\varphi_{\mu_{k-1}}(\omega^k))\zeta^k = 
	\varepsilon_1^{k-1}, \label{KKT1_k}\\
	& \nabla\varphi_{\mu_{k-1}}(\omega^k)^T \zeta^k - \eta_1^k = 
	\varepsilon_2^{k-1}, 
	\label{KKT2_k}\\
	& \nabla R_j(\omega^k)^T \zeta^k - \eta_j^k = 
	(\varepsilon_3^{k-1})_{\alert{j-1}} \ 
	(j=2,3, \dots, r), \label{KKT3_k}\\
	& \nabla_{\omega}G(\omega^k, \bar{\lambda}^k) + 
	\lambda_1^k\nabla\varphi_{\mu_{k-1}}(\omega^k) = \varepsilon_4^{k-1}, 
	\label{KKT4_k}\\
	& \lambda^k- \epsilon e_1 \geq 0, \quad \eta^k \geq 0 , \quad  (\lambda^k- 
	\epsilon e_1)^T \eta^k = 
	\varepsilon_5^{k-1}, 
	\label{KKT5_k} 
	\end{eqnarray}
	for all $k$. }

\medskip 
In \cite{OTKW21}, the authors proved that when $\psi (t) = t$ and $\phi (\mu,x) 
= \sqrt{x^2+\mu^2}$ and under Assumptions~(B1)-(B3), there exists some $\gamma 
>0$ such that
\begin{equation}
\mu_{k-1}^2 \geq \gamma \lvert\omega_j^k\rvert^{\frac{2}{2-p}} \quad (j\in 
I(\omega^*))
\label{eq:okunoineq}
\end{equation}
for all sufficiently large $k\in K$, where $K$ is the subsequence described in 
(i). This result is especially important in proving results related 
to (ii)-\alert{(iii)} above. 
However, in order to derive inequality\,\eqref{eq:okunoineq}, 
\cite{OTKW21} takes advantage of the specific function 
$\phi(\mu,x)$ chosen, which is not the case in the present work. 
Nevertheless, we have found out that such a strong result is not necessarily 
required to prove (ii)-\alert{(iii)}. In particular, it suffices to establish 
(i), which is indeed a weaker property. To this end, Assumption~(B3) will play 
a very significant role without which the analysis becomes extremely difficult.

\medskip

\medskip 
We now prove our first lemma which establishes property (i). 

\begin{lemma}\label{lemma:boundedratio}
	\alert{ Suppose that \alert{Assumptions~(B1)-(B3) hold}. Let 
	$(\omega^*,\lambda^*)$ be an arbitrary accumulation point of the sequence 
	$\{(\omega^k, \lambda^k)\}$ generated by Algorithm~\ref{algorithm}, and let 
	$\{(\omega^k, \lambda^k)\}_{k\in K}$ be an arbitrary subsequence converging 
	to $(\omega^*,\lambda^*)$. For any $j\in I(\omega ^*)$, $\{S_j^k\}_{k\in 
	K}$ is bounded. Moreover, $S_j^k \to 0$ as $k\in K \to \infty$ if 
	$p\in (0,1)$.}
\end{lemma}
\begin{proof}
	Denote
	\begin{equation}\label{Partial-G}
	F_j (\omega ^k, \bar{\lambda}^k) : = \frac{\partial G(\omega ^k, 
	\bar{\lambda}^k)}{\partial \omega _j}.
	\end{equation}
	\alert{From equations \eqref{KKT4_k}} and \eqref{Gradient-varphi}, we have
	\begin{equation}\label{KKT4-2}
	F_j (\omega ^k , \bar{\lambda}^k) + p \lambda_1 ^k \psi'\left([\phi (\mu 
	_{k-1}, \omega _j^k)]^{p}\right)\phi'(\mu _{k-1}, \omega _j ^k) [\phi (\mu 
	_{k-1}, \omega _j^k)]^{p-1} = (\varepsilon _4^{k-1})_j .
	\end{equation}
	\noindent \textbf{Case 1.} Suppose that $p=1$. Then 
	\begin{equation}\label{lemma:boundedratio-1}
	F_j (\omega ^k , \bar{\lambda}^k) +  \lambda_1^k \psi'(\phi (\mu _{k-1}, 
	\omega _j^k))\phi'(\mu _{k-1}, \omega _j ^k) = (\varepsilon _4^{k-1})_j .
	\end{equation}
	Rearranging the terms and using Assumptions~(A) and (B3), there are 
	constants 
	$c,r>0$ such that for sufficiently large $k$,
	\[
	\frac{\lvert(\varepsilon _4^{k-1})_j - F_j (\omega ^k , \bar{\lambda}^k) 
	\rvert }{\lambda _1 ^k\psi'(\phi (\mu _{k-1}, \omega _j^k))} = \lvert \phi 
	' 
	(\mu _{k-1}, \omega _j ^k)\rvert = 2 \int _0 ^{S_j^k} \rho (s)~ds\geq 1 - 
	\frac{c}{(S_j^k)^r + c},  \]
	where the second equality holds by \eqref{phi_firstder}, and 
	$\lambda_1^k>0$ 
	for sufficiently large $k$ \alert{since $\lambda_1^* \geq \epsilon > 0$}. 
	Consequently, we get
	\[
	\begin{array}{rl}
	\ds \frac{c}{(S_j^k)^r + c} & \ds \geq 1 - \frac{\lvert(\varepsilon 
	_4^{k-1})_j - F_j (\omega ^k , \bar{\lambda}^k) \rvert }{\lambda _1 
	^k\psi'(\phi (\mu _{k-1}, \omega _j^k))}  \\
	\ds &  \ds 	=\frac{\lambda _1^k\psi'(\phi (\mu _{k-1}, \omega _j^k)) - 
	\lvert(\varepsilon _4^{k-1})_j - F_j (\omega ^k , \bar{\lambda}^k) 
	\rvert}{\lambda_1^k\psi'(\phi (\mu _{k-1}, \omega _j^k))}  
	\end{array}
	\]
	Note that $\lambda _1^k\psi'(\phi (\mu _{k-1}, \omega _j^k)) - 
	\lvert(\varepsilon _4^{k-1})_j - F_j (\omega ^k , \bar{\lambda}^k) \rvert> 
	0$ 
	for all large $k$ \alert{by using the fact that 
	\eqref{eq:interiortosubdifferential} holds at $(\w^*,\lambda^*)$ when 
	$p=1$ and by invoking} Proposition 
	\ref{properties-smoothabsolutevalue}(c) and \eqref{lemma:boundedratio-1}. 
	Then 
	\begin{align*}
	0\leq  \frac{(S_j^k)^r}{c} \leq \frac{(S_j^k)^r + c}{c} &\leq 
	\frac{\lambda_1^k\psi'(\phi (\mu _{k-1}, \omega _j^k))}{\lambda 
	_1^k\psi'(\phi 
	(\mu _{k-1}, \omega _j^k)) - \lvert(\varepsilon _4^{k-1})_j - F_j (\omega 
	^k , 
	\bar{\lambda}^k) \rvert}\\
	& \to \frac{\lambda_1^*\psi'(0) }{\lambda_1^*\psi'(0) - \lvert F_j(\omega 
	^*, 
	\bar{\lambda}^*)\rvert} ~\text{as}~k\in K\to \infty
	\end{align*}
	where the finiteness of the limit is guaranteed by 
	\alert{\eqref{eq:interiortosubdifferential}}. Hence, 
	it 
	easily follows that $\{S_j^k\}_{k\in K}$ is bounded.
	
	\medskip 
	\noindent \textbf{Case 2.} Now, suppose $p\in (0,1)$. 
	From equation \eqref{KKT4-2}, we have 
	\[ 
	\begin{array}{rl}
	\ds \frac{\lvert(\varepsilon _4^{k-1})_j - F_j (\omega ^k , 
	\bar{\lambda}^k) \rvert }{p\lambda_1^k\psi'(\phi (\mu _{k-1}, \omega 
	_j^k))} = \lvert\phi ' (\mu _{k-1}, \omega _j ^k)\rvert \cdot  [\phi (\mu 
	_{k-1}, \omega _j^k)]^{p-1}.
	\end{array}	
	\]
	Using Assumptions~(B1) and (B3) and by \eqref{phi_firstder}, we get
	\[ 
	\ds  [\phi (\mu _{k-1}, \omega _j^k)]^{1-p} \cdot \frac{\lvert(\varepsilon 
	_4^{k-1})_j - F_j (\omega ^k , \bar{\lambda}^k) \rvert 
	}{p\lambda_1^k\psi'(\phi (\mu _{k-1}, \omega _j^k))} = \lvert\phi ' (\mu 
	_{k-1}, \omega _j ^k)\rvert \geq 1-\frac{c}{(S_j^k)^r + c} \geq 0.
	\]
	Meanwhile, note that since $1-p>0$ and $\omega _j^k \to 0$ as $k\in K\to 
	\infty$, we have 
	\[
	\lim _{k \in K\to\infty} [\phi (\mu _{k-1}, \omega _j^k)]^{1-p} = 0.
	\]
	Since $\lambda_1^{\ast}>0$ 
	%by Assumption~(B1) 
	and $\psi'(0) > 0$ by Assumption~(A), then
	\[\lim _{k\in K\to \infty} \frac{\lvert(\varepsilon _4^{k-1})_j - F_j 
	(\omega ^k , \bar{\lambda}^k) \rvert }{p\lambda_1^k\psi'(\phi (\mu _{k-1}, 
	\omega _j^k))}  = \frac{\lvert F_j(\omega ^*, 
	\bar{\lambda}^*)\rvert}{p\lambda _1^*\psi'(0)}.\]
	Thus,
	\[
	\lim_{k\in K\to \infty}\frac{c}{(S_j^k)^r + c} = 1.
	\]
	It follows that $\lim_{k\in K\to \infty}S_j^k = 0$, as desired.
	This completes the proof. 	
\end{proof}

We now focus on the sequence $\left\lbrace 
\nabla^2\varphi_{\mu_{k-1}}(\omega^k) \right\rbrace _{k\in K}$. Using Lemma 
\ref{lemma:boundedratio}, we prove the following important result, which will 
later be used 
to establish the boundedness of the sequence $\{ \zeta ^k, \eta ^k\}$ generated 
by Algorithm~\ref{algorithm}.

%-----------------------------------------------------------------------------------
% Proposition 4.3 to Lemma 3.4
\begin{lemma}\label{lemma:infinite_limit}
	Suppose that Assumptions~(B1)-(B4) hold. Let 
	$(\omega^*,\lambda^*)$ be an arbitrary accumulation point of the sequence 
	$\{(\omega^k, \lambda^k)\}$ generated by Algorithm~\ref{algorithm} and let 
	$\{(\omega^k, \lambda^k)\}_{k\in K}$ be an arbitrary subsequence converging 
	to $(\omega^*,\lambda^*)$. Then
	\[
	\lim_{k\in K \to \infty}\lvert(\nabla^2 
	\varphi_{\mu_{k-1}}(\omega^k))_{jj}\rvert = \infty \; \mbox{for}\; j \in 
	I(\omega^*).
	\] 
\end{lemma}
\begin{proof}
	 We first consider the case when $p=1$. In this instance, we have from 
	equation \eqref{Hessian-matrix} \alertchieu{that 
	\begin{equation*}\label{H-S-p=1}
	(\nabla^2\varphi_{\mu_{k-1}}(\omega^k))_{jj} = \psi''\left(\phi(\mu_{k-1}, 
	\omega_j^k)\right)\phi'(\mu_{k-1}, \omega_j^k)^2+ 
	\psi'\left(\phi(\mu_{k-1}, 
	\omega_j^k)\right) \phi''(\mu_{k-1}, \omega_j^k).
	\end{equation*}
	It follows from Assumption~(A) and Proposition 
	\ref{properties-smoothabsolutevalue}(c) that
	\[
		\psi''\left(\phi(\mu_{k-1}, \omega_j^k)\right)\phi'(\mu_{k-1}, 
	\omega_j^k)^2 \geq -\beta
	\]
	and there exists $\gamma>0$ such that for sufficient large $k$,
	\[
		\psi'\left(\phi(\mu_{k-1}, \omega_j^k)\right) \geq \gamma >0.
	\]
	Hence, for sufficiently large $k$ we obtain
	\[
		(\nabla^2\varphi_{\mu_{k-1}}(\omega^k))_{jj}\geq -\beta + 
	\gamma\phi''(\mu_{k-1}, \omega_j^k) = -\beta 
	+\gamma\frac{2}{\mu_{k-1}}\rho\left(\frac{\omega^k_j}{\mu_{k-1}}\right),
	\]
	where $\phi''(\mu_{k-1}, \omega_j^k)$ is defined as \eqref{phi_secondder}.}
	
	 By Lemma~\ref{lemma:boundedratio}, there exists $M>0$ such that 
	$\frac{\lvert\omega_j^k\rvert}{\mu _{k-1}}\leq M$ for all $k\in K$. Since 
	$\rho$ is nonincreasing on $[0,\infty)$ by Assumption~(B2), we have
	\alertchieu{ \begin{align*}
	 \lim _{k\in K \to 
	\infty}(\nabla^2\varphi_{\mu_{k-1}}(\omega^k))_{jj} 
	&\geq 
	-\beta +\lim _{k\in K \to \infty} 
	\frac{2\gamma}{\mu_{k-1}}\rho\left(\frac{\omega^k_j}{\mu_{k-1}}\right)\\
	 &\geq -\beta + \lim _{k\in K \to \infty} \frac{2\gamma}{\mu_{k-1}}\rho 
	(M) = 
	\infty ,
	\end{align*} where} the rightmost equality holds since  $\rho (t)>0$ on 
	$\Re$ by Assumption~(B4). This proves the claim for $p=1$. 
	 \medskip
	 \noindent 
	 
	 We now consider the case when $0<p<1$. 
	We look at two disjoint subsets of $K$:
	\[
	U^j_1 \coloneqq \{k\in K ~\lvert~ \omega_j^k=0\}, \quad \text{and} \quad  
	U^j_2 \coloneqq 
	\{k\in K ~\lvert~ \omega_j^k\ne 0\},
	\]
	and the corresponding subsequences. For $k\in U^j_1$,  we get from 
	\eqref{phi_firstder} that $\phi'(\mu_{k-1}, 0) = 0$. From 
	\eqref{Hessian-matrix}, we have
	\begin{equation}\label{lemma:infinite_limit.2}
	(\nabla^2\varphi_{\mu_{k-1}}(\omega^k))_{jj}  = p 
	\psi'\left([\phi(\mu_{k-1}, \omega^k_j)]^{p}\right)[\phi(\mu_{k-1}, 
	\omega^k_j)]^{p-1} \phi''(\mu_{k-1}, \omega^k_j).
	\end{equation}
	Meanwhile, from \eqref{phi_secondder}, 
	\begin{equation}\label{lemma:infinite_limit.3}
	\lim_{k\in U^j_1\to \infty} \phi''(\mu_{k-1}, \omega^k_j) = \lim_{k\in 
	U^j_1 \to 
	\infty}\frac{2}{\mu_{k-1}}\rho\left(\frac{\omega^k_j}{\mu_{k-1}}\right) = 
	\lim_{k\in U^j_1 \to \infty}\frac{2}{\mu_{k-1}}\rho(0)=\infty,
	\end{equation}
	where we note that $\rho (0)>0$ by Assumption~(B1) and the definition of 
	density function. Moreover, it is clear that
	\begin{equation}\label{lemma:infinite_limit.4}
	\lim_{k\in U^j_1\to \infty}[\phi(\mu_{k-1}, \omega^k_j)]^{p-1} = \infty.
	\end{equation}
	It follows from \eqref{lemma:infinite_limit.2}, 
	\eqref{lemma:infinite_limit.3}, \eqref{lemma:infinite_limit.4} and 
	Assumption~(A) that 
	\[\lim_{k\in U^j_1\to \infty} \lvert 
	(\nabla^2\varphi_{\mu_{k-1}}(\omega^k))_{jj}\rvert = \infty.\]
	
	\noindent For $k\in U^j_2$, we obtain
	\begin{equation}\label{lemma:infinite_limit.5}
	\lim_{k\in U^j_2\to \infty} \phi''(\mu_{k-1}, \omega^k_j) = 
	\lim_{k\alert{\in U^j_2}\to \infty}\frac{2}{\mu_{k-1}}\rho 
	\left(\frac{\omega^k_j}{\mu_{k-1}}\right)= \infty
	\end{equation}
	by using equation \eqref{phi_secondder} and the facts that $\lim_{k\in 
	U^j_2 \to \infty}\frac{\lvert\omega^k_j\rvert}{\mu_{k-1}} = 0$ from Lemma 
	\ref{lemma:boundedratio} and $\rho$ is continuous by Assumption~(B2).

	Meanwhile, \tblue{we obtain} $\phi''(\mu_{k-1},\omega_j^k) \neq 0$ for sufficiently large $k$ 
	using Lemma~\ref{lemma:boundedratio}, equation \eqref{phi_secondder} and 
	Assumption 
	(B1). Thus, invoking Assumption~(A), we have for large $k\in U_2^j$ that  
	\begin{align}
	&\psi''\left([\phi(\mu_{k-1}, \omega^k_j)]^p\right)\frac{[\phi'(\mu_{k-1}, 
	\omega^k_j)]^2[\phi(\mu_{k-1}, \omega^k_j)]^{p-1}}{\phi''(\mu_{k-1}, 
	\omega^k_j)} \notag{}\\
	&=  \psi''\left([\phi(\mu_{k-1}, 
	\omega^k_j)]^p\right)\frac{[\phi'(\mu_{k-1}, \omega^k_j)]^2[\phi(\mu_{k-1}, 
	\omega^k_j)]^{p}}{\phi''(\mu_{k-1}, \omega^k_j)\phi(\mu_{k-1}, \omega^k_j)} 
	\notag{} \\ 
	&\geq -\beta\frac{[\phi'(\mu_{k-1}, \omega^k_j)]^2[\phi(\mu_{k-1}, 
	\omega^k_j)]^{p}}{\phi''(\mu_{k-1}, \omega^k_j)\phi(\mu_{k-1}, 
	\omega^k_j)}. \label{psi''>=-beta}
	\end{align}
	Using \eqref{phi_secondder} \tblue{again, the symmetry of $\rho$,} and Proposition 
	\ref{properties-smoothabsolutevalue}(b), we have
	\[
	\phi''(\mu_{k-1}, \omega^k_j)\phi(\mu_{k-1}, \omega^k_j) \geq 
	\frac{2}{\mu_{k-1}}\rho \left(\frac{\omega^k_j}{\mu_{k-1}}\right) 
	\lvert\omega^k_j\rvert  =  2S_j^k\rho(S_j^k).
	\]
	Using this fact together with \eqref{psi''>=-beta} and equation 
	\eqref{phi_firstder}, we obtain
	\begin{align}
	&\psi''\left([\phi(\mu_{k-1}, \omega^k_j)]^p\right)\frac{[\phi'(\mu_{k-1}, 
	\omega^k_j)]^2[\phi(\mu_{k-1}, \omega^k_j)]^{p-1}}{\phi''(\mu_{k-1}, 
	\omega^k_j)} \notag{}\\
	&\geq -\beta\frac{2\left[\int_{0}^{S_j^k}\rho(t)dt\right]^2[\phi(\mu_{k-1}, 
	\omega^k_j)]^{p}}{S_j^k\rho(S_j^k)} \notag{} \\
	&\geq -\beta \frac{2S_j^k(\rho(0))^2[\phi(\mu_{k-1}, 
	\omega^k_j)]^{p}}{\rho(S_j^k)} ~\to ~0 ~\mbox{as} ~ k\in U^j_2 \to \infty, 
	\label{lemma:infinite_limit.6}
	\end{align}
	where the last inequality holds since 
	$\left[\int_{0}^{S_j^k}\rho(t)dt\right]^2 \leq (S_j^k)^2(\rho(0))^2$. 
	Similarly, we also have
	\begin{align}
	(p-1) \frac{[\phi'(\mu_{k-1}, \omega^k_j)]^2}{\phi''(\mu_{k-1}, \omega^k_j) 
	\phi(\mu_{k-1}, \omega^k_j)} & \geq   (p-1) 
	\frac{2\left[\int_{0}^{S_j^k}\rho(t)dt\right]^2}{S_j^k\rho(S_j^k)} 
	\notag{}\\
	& \geq (p-1) \frac{2S_j^k(\rho(0))^2}{\rho(S_j^k)} ~\to ~0 ~\mbox{as} ~ 
	k\in U^j_2 \to \infty. \label{lemma:infinite_limit.7}
	\end{align}
	
	\noindent For brevity, denote $\phi_k \coloneqq \phi(\mu_{k-1}, 
	\omega^k_j)$, 
	$\phi'_k \coloneqq \phi'(\mu_{k-1}, \omega^k_j)$, $\phi''_k \coloneqq 
	\phi''(\mu_{k-1}, 
	\omega^k_j)$. Then \eqref{Hessian-matrix} can be written as
	\begin{equation}\label{hessian_U2}
	(\nabla^2\varphi_{\mu_{k-1}}(\omega^k))_{jj}= p\phi''_k 
	\phi_k^{p-1}\left[p\psi''\left(\phi_k^p\right)\frac{[\phi'_k]^2\phi_k^{p-1}}{\phi''_k}
	 + (p-1)\psi'\left(\phi_k^p\right) \frac{[\phi'_k]^2}{\phi''_k \phi_k} + 
	\psi'\left(\phi_k^p\right) \right].
	\end{equation}
	Since $0 <\psi'(t) \leq \alpha$ by Assumption~(A), we have
	\[
	\lim_{k\in U^j_2 \to 
	\infty}\left[p\psi''\left(\phi_k^p\right)\frac{[\phi'_k]^2\phi_k^{p-1}}{\phi''_k}
	 + (p-1)\psi'\left(\phi_k^p\right) \frac{[\phi'_k]^2}{\phi''_k \phi_k} + 
	\psi'\left(\phi_k^p\right)\right] >0
	\]
	by using the obtained limits \eqref{lemma:infinite_limit.6} and 
	\eqref{lemma:infinite_limit.7}. On the 
	other hand, it is clear from \eqref{lemma:infinite_limit.5} that $\phi''_k 
	\phi_k^{p-1} 
	\to \infty$ as $k\in U^j_2 \to \infty$. 
	Hence, taking the limit in \eqref{hessian_U2}, 
	\[
	\lim_{k\in U^j_2 \to \infty}(\nabla^2\varphi_{\mu_{k-1}}(\omega^k))_{jj} = 
	\infty.
	\]
	This completes the proof. 
\end{proof} 

% We next prove the following lemma which will be used to prove the boundedness 
%of the iterates $\{(\zeta ^k, \eta ^k) \}$. In this result, we do not actually 
%need $\{\omega ^k, \mu _{k-1} \}$ to be generated from the algorithm. 

\begin{lemma}\label{lemma:lagrange_multipliers}
	Suppose that Assumptions~(B1)-(B4) hold, and $(\omega^*, 
	\lambda^*, \zeta^*, \eta^*)$ is an accumulation point of the sequence 
	$\{(\omega^k, \lambda^k, 
	\zeta^k, 
	\eta^k)\}$ generated by Algorithm~\ref{algorithm}. Then
	\begin{itemize}
		\item [(i)] $\zeta_j^* = 0$ for all $j\in I(\omega^*)$, that is, 
		$\hatzeta^* = 0$; and
		
		\item [(ii)] $\sum_{j\notin I(\omega^*)}{\rm 
			sgn}(\omega^*_j)\lvert\omega^*_j\rvert^{p-1}\psi'(\lvert\omega^*_j\rvert^p)\zeta^*_j=
		\eta_1^*$,
	\end{itemize}
\end{lemma}
\alert{
\begin{proof}
	Let $\{(\omega^k, \lambda^k, \zeta^k, \eta^k)\}_{k\in K}$ be a subsequence 
	converging to $(\omega^*, \lambda^*, \zeta^*, \eta^*)$. From 
	\eqref{KKT1_k}, 
	we have for all $k\in K$ that 
	\[ (\nabla f(\omega^k))_j + 
	\left(\nabla^2_{\omega\omega}G(\omega^k,\bar{\lambda}^k)\zeta^k\right)_j
	+ \lambda_1^k(\nabla^2\varphi_{\mu_{k-1}}(\omega^k))_{jj}{\zeta}^k_j = 
	(\varepsilon_1^{k-1})_j ,\ (j=1,2,\dots,n).\]
	Since $G$ is twice continuously differentiable and $f$ is continuously 
	differentiable, then 
	$\{\lambda_1^k(\nabla^2\varphi_{\mu_{k-1}}(\omega^k))_{jj}\zeta^k_j\}_{k\in 
		K}$ is a bounded sequence for each $j$. Consequently, $\zeta_j^* = 0$ 
		for 
	all $j\in I(\omega^*)$, since $\lambda_1^*>0$ and $\lim_{k\in K \to \infty} 
	(\nabla^2\varphi_{\mu_{k-1}}(\omega^k))_{jj} = +\infty$ for each $j\in 
	I(\omega^*)$ by Lemma~\ref{lemma:infinite_limit}. This proves part (i). 
	
	To prove part (ii), we note from \eqref{KKT4_k} that for all $k\in K$,
	\[F_j(\omega^k,\bar{\lambda}^k) + 
	\lambda_1^k (\nabla \varphi_{\mu_{k -1}}(\omega^k) )_j = 
	(\varepsilon_4^{k-1})_j ,\ (j=1,2,\dots,n),\]
	so that $\{ (\nabla \varphi_{\mu_{k -1}}(\omega^k) )_j\}_{k\in K}$ is 
	convergent since 
	$F_j$ given by \eqref{Partial-G} is continuous and $\lambda_1^* > 0$. 
	Hence, we obtain 
	from item (i) that
	$\lim_{\alertchieu{k\in K\to\infty}}(\nabla\varphi_{\mu_{k-1}}(\omega^k))_j \zeta^k_j = 0$, 
	$(j\in I(\omega^*))$ and therefore,
	\begin{equation*}
	\lim_{k\in K\to\infty}\sum_{j\in 
		I(\omega^*)}(\nabla\varphi_{\mu_{k-1}}(\omega^k))_j \zeta^k_j = 0.
	\end{equation*}
	Together with \eqref{prop41.1} and \eqref{KKT2_k}, it follows that
	\begin{align*}
	\eta_1^* = \lim_{k\in K\to \infty} \eta_1^k & = \lim_{k\in K\to 
		\infty}\nabla\varphi_{\mu_{k-1}}(\omega^k)^T \zeta^k \\
	&= \lim_{k\in K\to \infty}\sum_{j\notin 
		I(\omega^*)}(\nabla\varphi_{\mu_{k-1}}(\omega^k))_j \zeta^k_j\\
	&= \sum_{j\notin I(\omega^*)}p{\rm 
		sgn}(\omega^*_j)\lvert\omega^*_j\rvert^{p-1}\psi'(\lvert\omega^*_j\rvert^p)
	\zeta^*_j,
	\end{align*}
	proving part (ii). 
\end{proof}

\begin{remark}\label{rem:liminf}
	Crucial in the proofs of Lemmas 
	\ref{lemma:boundedratio}-\ref{lemma:lagrange_multipliers} is the strict 
	positivity of the regularization parameter $\lambda_1^*$ that corresponds 
	to \tblue{an accumulation} point of $\{\lambda_1^k\}$. This is automatically guaranteed by 
	the constraint set \eqref{eq:Omega_eps}, as opposed to the work of 
	\cite{OTKW21} where the parameter $\epsilon$ is set to $0$. In turn, 
	\cite{OTKW21} requires the assumption that $\liminf_{k\to\infty} 
	\lambda_1^k > 0$ to ensure that $\lambda_1^*>0$, but such an assumption is 
	difficult to guarantee for the iterates generated by Algorithm 
	\ref{algorithm}. 
\end{remark}

Having derived all the necessary lemmas, we can now prove our main result. 

\begin{proof}[Proof of Theorem~{\upshape\ref{thm:convergence-analysis}}]
	Let $\{(\omega^k, \lambda^k, \zeta^k, \eta^k)\}_{k\in K}$ be a subsequence 
	converging to an accumulation point $(\omega^*, \lambda^*, \zeta^*, 
	\eta^*)$. 
	It is clear from \eqref{KKT3_k} and \eqref{KKT5_k} that equations 
	\eqref{bkkt5} and \eqref{bkktlast} hold. Meanwhile, we obtain from 
	\eqref{KKT1_k} and \eqref{KKT4_k}, respectively, that
	\begin{eqnarray}
	& \nabla_{\tomega} f(\omega^k) + (\nabla^2_{\tomega\tomega}G(\omega^k, 
	\bar{\lambda}^k)\tzeta^k + (\nabla^2_{\tomega\hatomega}G(\omega^k, 
	\bar{\lambda}^k)\hatzeta^k+
	\lambda_1^k\nabla_{\tomega}^2\varphi_{\mu_{k-1}}(\omega^k))\tzeta^k = 
	\tvarepsilon_1^{k-1}, \label{KKT1_k2}\\
	& \nabla_{\tomega}G(\omega^k, \bar{\lambda}^k) + 
	\lambda_1^k\nabla_{\tomega}\varphi_{\mu_{k-1}}(\omega^k) = 
	\tvarepsilon_4^{k-1}, 
	\label{KKT4_k2}
	\end{eqnarray}
	where $\tvarepsilon_1^{k-1} = \{(\varepsilon_1^{k-1})_j\}_{j\notin 
		I(\omega^*)}$ and $\tvarepsilon_4^{k-1} = 
	\{(\varepsilon_4^{k-1})_j\}_{j\notin 
		I(\omega^*)}$. Using Lemma~\ref{lemma:limit_gradient_hessian} and Lemma 
		\ref{lemma:lagrange_multipliers}(i), and 
		letting 
	$k\in K\to \infty$ in \eqref{KKT1_k2} and \eqref{KKT4_k2}, we obtain the 
	bilevel KKT conditions \eqref{bkkt1} and \eqref{bkkt2}. Finally, 
	\eqref{bkkt3} and \eqref{bkkt4} hold by Lemma 
	\ref{lemma:lagrange_multipliers}. This completes 
	the proof of 
	Theorem~\ref{thm:convergence-analysis}.

\end{proof}
}

\subsection{\alert{Boundedness}}
\alert{In the preceding discussion, we have shown that accumulation points of 
$\{(\omega^k,\lambda^k,\zeta^k,\eta^k)\}$ correspond to bilevel KKT points. The 
existence of these accumulation points is guaranteed by boundedness of the full 
sequence $\{(\omega^k,\lambda^k,\zeta^k,\eta^k)\}$. In this section, we show 
that the boundedness of  $\{(\omega^k,\lambda^k)\}$ is enough to conclude that 
the full sequence is bounded. 
  \subsubsection{\alert{Weaker constraint qualification}}
  In \cite{OTKW21}, linearly independent 
  constraint qualification (LICQ) was one of the assumptions used to obtain the 
  boundedness of the sequence $\{(\omega^k,\lambda^k,\zeta^k,\eta^k)\}$. In 
  this present work, we only assume that the 
Mangasarian-Fromovitz constraint qualification holds at accumulation points 
of a sequence generated by the smoothing algorithm. }
\medskip 

\noindent \textbf{\alert{Assumption~(C).}} Let $(\omega^*,\lambda^*)\in \Re^n 
\times 
\Re^r$ be an accumulation point of $\{(\omega^k,\lambda^k)\}$ generated by 
Algorithm~\ref{algorithm}. Denote
\[
I(\lambda^*) \coloneqq \{i \in \{1, 2, \dots, r\}\ \lvert \ \lambda^*_i = 
\tblue{\epsilon}\alert{(e_1)_i},\}
\] 
\alert{where $e_1 = (1,0,\dots,0)\in \Re^r$}, and
\[
\Phi_j(\omega, \lambda)\coloneqq \frac{\partial G(\omega, 
\bar{\lambda})}{\partial 
\omega_j} + p {\rm 
sgn}(\omega_j)\lambda_1\lvert\omega_j\rvert^{p-1}\psi'(\lvert\omega_j\rvert^p) 
\;  (j \notin I(\omega^*)).
\] Then, the Mangasarian-Fromovitz constraint qualification (MFCQ) holds at 
$(\omega, \lambda) = (\omega^*,\lambda^*)$ for the constraints
$\Psi_j(\omega,\lambda) = 0$ for all $j=1,\dots, n$ and $\lambda\geq 0$, where 
\[
\Psi_j(\omega, \lambda) \coloneqq \left\{
\begin{array}{lc}
\Phi_j(\omega, \lambda)& {\rm if}\, j\notin I(\omega^*),\\ 
\omega_j& {\rm if}\, j\in I(\omega^*).\\
\end{array}
\right.
\]
That is, $\{ \nabla _{(\omega,\lambda)}\Psi _j (\omega^*,\lambda^*) \}_{j=1}^n$ 
is linearly independent and there exists $\bar{d}\in \Re^{n+r}$ such that
\begin{eqnarray}
\nabla _{(\omega,\lambda)}\Psi _j (\omega^*,\lambda^*) ^T \bar{d} = 0  \quad 
\forall j=1,\dots, n, \label{MFCQ1} \\ 
(\nabla_{(\omega,\lambda)}\lambda_i 
\lvert_{(\omega,\lambda)=(\omega^*,\lambda^*)})^T\bar{d}>0 \qquad \forall i\in 
I(\lambda^*). \label{MFCQ2}
\end{eqnarray}
The following 
lemma is needed for subsequent analysis.
%-------------------------------------------------------------------------------------------
% Lemma 3.1
\begin{lemma}\label{propertyMFCQ}
	Suppose that $(\omega^*,\lambda^*)$ is an arbitrary accumulation point of 
	the 
	sequence $\{(\omega^k, \lambda^k)\}$ such that Assumption \alert{(C)} 
	holds. 
	Then,  $\{\nabla_{(\tomega,\lambda)}\Phi_j(\omega^*, 
	\lambda^*)\}_{j\notin I(\omega^*)}$ is linearly independent and there 
	exists a 
	vector $d\in \Re^{n-\lvert I(\omega^*)\rvert+r}$ such that
	\begin{eqnarray}
	\nabla_{(\tomega,\lambda)}\Phi_j(\omega^*, \lambda^*)^Td=0 \quad  
	\forall j\notin I(\omega^*), \label{MFCQ1_1} \\ 
	(\nabla_{(\tomega,\lambda)}\lambda_i 
	\lvert_{(\omega,\lambda)=(\omega^*,\lambda^*)} )^T d>0 \quad \forall  i\in 
	I(\lambda^*), \label{MFCQ2_2}
	\end{eqnarray}
	where $\tomega \coloneqq (\omega_j)_{j\alert{\notin }I(\omega^*)}$. 
\end{lemma}
\begin{proof}
	See Appendix\,\ref{sec:appB}.
\end{proof}

\subsubsection{Boundedness of algorithm iterates}

We will now show the boundedness of the sequence of Lagrange multiplier vectors 
$\{(\zeta^k, \eta^k)\}$ in the following lemma. 
%The proof follows some of the
%ideas used in \cite[Proposition 9]{OTKW21}. However, there are important 
%arguments that differ from those in \cite[Proposition 9]{OTKW21}, which we 
%highlight in the following proof. 
%-----------------------------------------------------------------------------------------------
% Proposition 4.4 to Lemma 3.6
\begin{proposition}\label{prop44}
	Suppose that Assumptions~(B) and (C) hold. Let $\{(\zeta^k, 
	\eta^k)\} \subseteq \Re^n \times \Re^r$ be a sequence of the accompanying 
	Lagrange multiplier vectors generated by Algorithm~\ref{algorithm}. 
	\alert{If $\{\omega^k,\lambda^k)\}$ is bounded,} then 
	$\{(\zeta^k, \eta^k)\}$ is bounded. 
\end{proposition}

\begin{proof}
	For convenience, denote
	\[
	\xi^k \coloneqq ((\zeta^k)^T, (\eta^k)^T)^T, \quad \hat{\zeta}^k 
	\coloneqq\frac{\zeta^k}{\|\xi^k\|},\quad  \hat{\eta}^k \coloneqq 
	\frac{\eta^k}{\|\xi^k\|}
	\]
	for each $k$. We prove by contradiction that the sequence 
	$\{(\alert{\zeta}^k, \eta^k)\}$ is bounded. Without loss of generality, we 
	may assume that
	\begin{equation*}\label{prop44.1}
	\|\xi^k\| \to \infty, \ \lim_{k\to \infty}\frac{\xi^k}{\|\xi^k\|}= 
	\hat{\xi}^*,
	\end{equation*}
	where $\hat{\xi}^*\coloneqq((\alert{\hat{\zeta}}^*)^T, 
	(\alert{\hat{\eta}}^*)^T)^T$ with $\alert{\hat{\zeta}}^*$ and 
	$\alert{\hat{\eta}}^*$ are accumulation points of $\{\hat{\zeta}^k\}$ and 
	$\{\hat{\eta}^k\}$, respectively. 
%	By Assumption~(B2), 
	\alert{We} may suppose 
	without loss of generality that $\lim_{k\to \infty}(\omega^k, \lambda^k) = 
	(\omega^*, \lambda^*)$. Dividing by $\|\xi^k\|$ both sides of 
	\eqref{KKT1},\eqref{KKT2},\eqref{KKT3} and \eqref{KKT5} evaluated at 
	$(\omega,\lambda,\zeta,\eta) = (\omega^k,\lambda^k,\zeta^k,\eta^k)$ and 
	$(\varepsilon_1,\varepsilon_2,\varepsilon_3,\varepsilon_4,\varepsilon_5) = 
	(\varepsilon^{k-1}_1,\varepsilon^{k-1}_2,\varepsilon^{k-1}_3,\varepsilon^{k-1}_4,\varepsilon^{k-1}_5)$, we have for each $k$ the following equations:
%	\begin{equation}\label{prop44.2}
%	\frac{(\nabla f(\omega^k))_j}{\|\xi^k\|} + 
%	\left(\nabla^2_{\omega\omega}G(\omega^k,\bar{\lambda}^k)\hat{\zeta}^k\right)_j
%	 + \lambda_1^k(\nabla^2\varphi_{\mu_{k-1}}(\omega^k))_{jj}\hat{\zeta}^k_j = 
%	\frac{(\varepsilon_1^{k-1})_j}{\|\xi^k\|} \ (j=1,2,\dots,n),
%	\end{equation}
%	\begin{equation}\label{prop44.3}
%	\nabla\varphi_{\mu_{k-1}}(\omega^k)^T \hat{\zeta}^k - \hat{\eta}^k_1= 
%	\frac{\varepsilon_2^{k-1}}{\|\xi^k\|},
%	\end{equation} 
%	\begin{equation}\label{prop44.4}
%	\nabla R_{j}(\omega^k)^T \hat{\zeta}^k - \hat{\eta}^k_j= 
%	\frac{(\varepsilon_3^{k-1})_{j-1}}{\|\xi^k\|} \ (j=2,3,\dots,r),
%	\end{equation}  
%	\begin{equation}\label{prop44.5}
%	\alert{\lambda_j^k -\tblue{\epsilon}(e_1)_j
%		\geq 0}, \quad  \hat{\eta}^k_j \geq 0 , \quad 
%		(\lambda_j^k-\tblue{\epsilon}(e_1)_j) 
%		\hat{\eta}^k_j 
%		\leq \sum_{\alert{l}=2}^r \alert{(\lambda_l^k-\tblue{\epsilon}(e_1)_l) }
%		\hat{\eta}^k_l = \frac{\varepsilon_5^{k-1}}{\|\xi^k\|}\quad 
%		(j=1,\dots,r).
%		
%		
%	\end{equation} 
	\begin{align}
&	\frac{(\nabla f(\omega^k))_j}{\|\xi^k\|} + 
	\left(\nabla^2_{\omega\omega}G(\omega^k,\bar{\lambda}^k)\hat{\zeta}^k\right)_j
	 + \lambda_1^k(\nabla^2\varphi_{\mu_{k-1}}(\omega^k))_{jj}\hat{\zeta}^k_j = 
	\frac{(\varepsilon_1^{k-1})_j}{\|\xi^k\|} \ (j=1,2,\dots,n),\label{prop44.2}\\
&	\nabla\varphi_{\mu_{k-1}}(\omega^k)^T \hat{\zeta}^k - \hat{\eta}^k_1= 
	\frac{\varepsilon_2^{k-1}}{\|\xi^k\|},\label{prop44.3}\\
&	\nabla R_{j}(\omega^k)^T \hat{\zeta}^k - \hat{\eta}^k_j= 
	\frac{(\varepsilon_3^{k-1})_{j-1}}{\|\xi^k\|} \ (j=2,3,\dots,r),\label{prop44.4}\\
&	\alert{\lambda_j^k -\tblue{\epsilon}(e_1)_j
		\geq 0}, \quad  \hat{\eta}^k_j \geq 0 , \quad 
		(\lambda_j^k-\tblue{\epsilon}(e_1)_j) 
		\hat{\eta}^k_j 
		\leq \sum_{\alertchieu{j=1}}^r \alert{(\lambda_j^k-\tblue{\epsilon}(e_1)_j) }
		\hat{\eta}^k_j = \frac{\varepsilon_5^{k-1}}{\|\xi^k\|}\quad 
		(j=1,\dots,r). \label{prop44.5}	%\\
%&\tblue{[[[ \mbox{In the last line, I think}
%	\sum_{\alert{l}=1}^r \alert{(\lambda_l^k-\tblue{\epsilon}(e_1)_l) 
%}\hat{\eta}^k_l
%	\mbox{ is correct} 
%		]]]}\notag
	\end{align} 
	
%	where $(\omega,\lambda,\zeta,\eta) = (\omega^k,\lambda^k,\zeta^k,\eta^k)$ 
%	and 
%	$(\varepsilon_1,\varepsilon_2,\varepsilon_3,\varepsilon_4,\varepsilon_5) = 
%	
%(\varepsilon^{k-1}_1,\varepsilon^{k-1}_2,\varepsilon^{k-1}_3,\varepsilon^{k-1}_4,\varepsilon^{k-1}_5)$.
	 Since $\lim_{k\to\infty}\frac{\varepsilon_l^{k-1}}{\|\xi^k\|} = 0$, 
	$l=1,2,3,5$, letting 
	$k\to \infty$ in inequality \eqref{prop44.5} 
	gives
%	\begin{equation}\label{prop44.6}
%	\hat{\eta}^*_1 =0, 
%	\end{equation}
%	and similarly, 
	\begin{equation}\label{prop44.7}
	\hat{\eta}^*_j =0 \ (j\notin I(\lambda^*)) ~~\mbox{and}~~ \hat{\eta}^*_j 
	\geq 0 \ (j\in I(\lambda^*)).
	\end{equation} 
	Since $\|\hat{\xi}^*\| = 1$, we get from \eqref{prop44.7} that
	\begin{equation}\label{prop44.8}
	1=\|\hat{\zeta}^*\|^2 + \|\hat{\eta}^*\|^2 = \|\hat{\zeta}^*\|^2 + 
	\sum_{j\in I(\lambda^*)}\lvert\hat{\eta}^*_j\rvert^2.
	\end{equation}
	Meanwhile, since $G$ is twice continuously differentiable and $f$ is 
	continuously differentiable, then both 
	$\{\nabla^2_{\omega\omega}G(\omega^k,\bar{\lambda}^k)\hat{\zeta}^k\}$ and 
	$\left\lbrace \frac{\nabla f(\omega^k)}{\|\xi^k\|}  \right\rbrace$ are 
	bounded. This implies the boundedness of 
	$\{\lambda_1^k(\nabla^2\varphi_{\mu_{k-1}}(\omega^k))_{jj}\hat{\zeta}^k_j\}$
	 for each $j$ by \eqref{prop44.2}. Consequently, we obtain $\lim_{k\to 
	\infty}\hat{\zeta}^k_j = 0$ for $j\in I(\omega^*)$ by 
%using Assumption~(B1) and 
\alert{Lemma~\ref{lemma:infinite_limit} and noting that $\lambda_1^* >0$}. That 
is,
	\begin{equation}\label{prop44.9}
	\hat{\zeta}^*_j = 0 \ (j\in I(\omega^*)).
	\end{equation}
%	By invoking Assumption~(B1) and 
	\alert{L}etting $k\to \infty$ in \eqref{KKT4}, it 
	is clear that 
	\begin{equation}\label{lim_nabla_varphi_mu}
	\lim_{k\to\infty}\lvert\nabla (\varphi_{\mu_{k -1}} (\omega ^k))_j\rvert = 
	\frac{\lvert F_j(\omega ^*, \bar{\lambda}^*)\rvert}{\lambda_1^*},
	\end{equation} 
	where $F_j(\omega ^*, \bar{\lambda}^*)$ is given by $\eqref{Partial-G}$. 
	This together with \eqref{prop44.9} gives us  
	\begin{equation}\label{prop44.10}
	\lim_{k\to\infty}\sum_{j\in 
	I(\omega^*)}(\nabla\varphi_{\mu_{k-1}}(\omega^k))_j \hat{\zeta}^k_j = 0.
	\end{equation}
%	On the other hand, we get 
%	\[\lim_{k\to \infty}\nabla\varphi_{\mu_{k-1}}(\omega^k)^T \hat{\zeta}^k = 
%	\alert{\hat{\eta}_1^*}\]
%	by letting $k\to \infty$ in \eqref{prop44.3}. 
%	Together with \eqref{prop44.10}, it yields
\alert{Thus, 
	\begin{align}
	\alert{\hat{\eta}_1^*}& \overset{\eqref{prop44.3}}{=}\lim_{k\to 
	\infty}\nabla\varphi_{\mu_{k-1}}(\omega^k)^T \hat{\zeta}^k \notag \\
&= 
	\lim_{k\to \infty}\left(\sum_{j\in 
	I(\omega^*)}(\nabla\varphi_{\mu_{k-1}}(\omega^k))_j \hat{\zeta}^k_j + 
	\sum_{j\notin I(\omega^*)}(\nabla\varphi_{\mu_{k-1}}(\omega^k))_j 
	\hat{\zeta}^k_j\right)   \notag \\
	&\overset{\eqref{prop44.10}}{=}\lim_{k\to \infty}\sum_{j\notin 
	I(\omega^*)}(\nabla\varphi_{\mu_{k-1}}(\omega^k))_j \hat{\zeta}^k_j  \notag 
	\\
	& \overset{\eqref{prop41.1}}{=} \sum_{j\notin I(\omega^*)}p{\rm 
	sgn}(\omega^*_j)\lvert\omega^*_j\rvert^{p-1}\psi'(\lvert\omega^*_j\rvert^p) 
	\hat{\zeta}^*_j. \label{eq:eta1*}
	\end{align}
On the other hand, we have from \eqref{prop44.4} and \eqref{prop44.9} that 
	\begin{equation}
		\hat{\eta}_j^* = \sum_{i\notin I(\w^*)} \frac{\partial 
		R_j(\w^*)}{\partial \w_i} 	\hat{\zeta}^*_i, \quad (j=2,\dots, r). 
		\label{eq:etaj*}
	\end{equation}
Meanwhile, letting $k\to\infty$ in \eqref{prop44.2} and using equations 
\eqref{prop44.9} and \eqref{prop41.2}, 
we have 
	\begin{equation}
		\left(\nabla^2_{\tomega\tomega}G(\omega^*, 
		\bar{\lambda}^*)+
		\lambda_1^*	
		p^2\psi''\left(\lvert\tomega^*_j\rvert^{p}\right)\lvert\tomega^*_j\rvert^{2p-2}
		+ 
		\lambda_1^*p(p-1)\psi'\left(\lvert\tomega^*_j\rvert^{p}\right)\lvert\tomega^*_j\rvert^{p-2}
		 \right) \tilde{\hat{\zeta}}^*  = 0, \label{eq:nablaphi_w}
	\end{equation}
where $ \tilde{\hat{\zeta}}^* = (\hat{\zeta}^*_i)_{i\notin I(\w^*)}$. Combining 
equations \eqref{eq:eta1*}, \eqref{eq:etaj*} and \eqref{eq:nablaphi_w}, we 
obtain}
	\begin{equation}\label{prop44.11}
	\sum_{j\notin 
	I(\omega^*)}\hat{\zeta}^*_j\nabla_{(\tomega,\lambda)}\Phi_{j}(\omega^*,\lambda^*)
	 - \sum_{j\in 
	I(\lambda^*)}\hat{\eta}^*_j\nabla_{(\tomega,\lambda)}\lambda_j(\omega^*,\lambda^*)
	 = 0,
	\end{equation}
	where $\tomega\coloneqq(\omega_j)_{j\notin I(\omega^*)}$ and 
	$\Phi_{j}(\omega,\lambda)$ $(j\notin I(\omega^*))$ are as defined in 
	\alert{Assumption~(C)}. On the other hand, by Lemma~\ref{propertyMFCQ}, we 
	can 
	find a vector $d\in \Re^{n-\lvert I(\omega^*)\rvert+r}$ such that 
	\eqref{MFCQ1_1} and \eqref{MFCQ2_2} hold. From \eqref{prop44.11}, we have 
	\begin{equation*}%\label{prop44.11}
	\sum_{j\notin 
	I(\omega^*)}\hat{\zeta}^*_j\nabla_{(\tomega,\lambda)}\Phi_{j}(\omega^*,\lambda^*)^T
	 d  - \sum_{j\in 
	I(\lambda^*)}\hat{\eta}^*_j\nabla_{(\tomega,\lambda)}\lambda_j(\omega^*,\lambda^*)^Td
	 = 0.
	\end{equation*}
	Together with equation \eqref{MFCQ1_1}, we obtain 
	\[
	\sum_{j\in 
	I(\lambda^*)}\hat{\eta}^*_j\nabla_{(\tomega,\lambda)}\lambda_j(\omega^*,\lambda^*)^Td=0.
	\]
	Consequently, we have from \eqref{prop44.7} and \eqref{MFCQ2_2} that 
	$\hat{\eta}_j^* = 0$ for all $j\in I(\lambda^*)$. In turn, 
	\eqref{prop44.11} implies that
	\[ \sum_{j\notin 
	I(\omega^*)}\hat{\zeta}^*_j\nabla_{(\tomega,\lambda)}\Phi_{j}(\omega^*,\lambda^*)
	 = 0.\]
	Since 
	$\{\nabla_{(\tomega,\lambda)}\Phi_{j}(\omega^*,\lambda^*)\}_{j\notin 
	I(\omega^*)}$ is linearly independent by Lemma~\ref{propertyMFCQ}, then 
	$\hat{\zeta}_j^* = 0$ for all $j\notin I(\omega^*)$. Together with 
	\eqref{prop44.9}, we have $\hat{\zeta}^*=0$ which in turn implies that 
	$\|\hat{\zeta}^*\|^2 + \sum_{j\in 
	I(\lambda^*)}\lvert\hat{\eta}^*_j\rvert^{\alert{2}} = 0$. This contradicts 
	\eqref{prop44.8}. Therefore, the sequence $\{(\zeta^k, \eta^k)\}$ is 
	bounded. 
\end{proof}

%By the above lemma, we know that accumulation points of $\{(\omega^k, 
%\lambda^k, \zeta^k, \eta^k)\}$ exist. We now show that such point satisfies 
%condition \eqref{SBKKT3}-\eqref{SBKKT4}. 
 \section{{Numerical results}} 
%
%\subsection{Simulations}\label{numerical}
 We compare the efficiency of different smoothing functions, namely the 
 functions $\phi_i\ (i=1,2,\ldots,6)$ presented in Appendix 
 \ref{app:examples_smoothing} by means of numerical simulation.
The program is coded in \alert{MATLAB~R2022b and run on a machine with Intel(R) 
Core(TM) i7-7500U CPU@2.70GHz and 8.0 GB RAM.}

\subsection{\alertsection{Problem with an Elastic-Net-type regularizer}}
We solve the following bilevel problem arising from squared linear regression using an Elastic-Net-type regularizer:
\begin{align}\label{al:0913}
\begin{array}{rll}
\displaystyle{\min_{\w,\lambda}}&\ &\displaystyle{\frac{1}{2}\|A_{1}{\w}-b_{1}\|_2^2} \\ 
\mbox{s.t. }&  &\w \in\displaystyle{\mathop{\rm argmin}_{\hat{\w}\in\Re^n}}
\left\{
\frac{1}{2}\|A_{2}\hat{\w}-b_{2}\|_2^2 + {\lambda_1\|\hat{\w}\|_p^p}
+\lambda_2\|\hat{\w}\|_2^2
\right\}\\ 
& &\lambda_1 \geq \epsilon , \quad \lambda_2 \geq 0, 
\end{array}
\end{align}
where
$A_{i}\in \Re^{m_{i}\times n}$, $b_{i}\in \Re^{m_{i}}$ for $i\in \{1,2\}$ and 
$\epsilon = 10^{-6}$.
We produce 20 synthetic problems for 
$(n,m_{1},m_{2})=(500,1000,1000)$ and for $(n,m_{1},m_{2})=(500,300,300)$ 
generated in Matlab as follows:
%The data matrices and vectors $A_{\{1,tr\}}, b_{\{1,t\}}$
%are generated by
%\texttt{rand}. We also generate
\begin{eqnarray*}
  &A_{i}\coloneqq\texttt{rand}(m_{i},n),\ 
  \begin{bmatrix}b_1\\b_{2}\end{bmatrix}\coloneqq\begin{bmatrix}A_1*\theta 
  \\A_2*\theta+0.01*\left(2*\texttt{rand}\left(m_{2},1\right)-\texttt{ones}\left(m_2,1\right)\right)\end{bmatrix},&\\
  &\theta\coloneqq\texttt{zeros}(n,1), \quad \theta 
  (\texttt{randsample}(n,0.15*n)) = -5+10*\texttt{rand}(0.15*n,1),&
\end{eqnarray*}
with \texttt{rand}, \texttt{randn}, \texttt{ones}, and \texttt{zeros} being 
MATLAB commands, 
and apply Algorithms~\ref{algorithm} with the smoothing functions $\phi_i\ 
(i=1,2,\ldots,6)$ to the problems \eqref{al:0913} with the generated data. The 
random number generator is initialized 
at \texttt{default}. The test data $A_3\in \Re^{m_3\times n}$ and $b_3 \in 
\Re^{m_3}$ are generated in the same manner as $A_i$ and $b_i$ for $i=1,2$, with 
$m_3 \coloneqq m_1$. 
%Moreover, the parameter $p$ in the $\ell_p$-regularizer is set to $p=1,0.5$ 
%for each problem.

%Morevoer, \texttt{rand} is run with 0.234 as a seed.
%[{\rm zeros}(\lfloor n \rfloor /3,1);]$  

In order to compute a KKT point of the smoothed subproblem for
\eqref{al:0913} in Step~1 of Algorithm\,\ref{algorithm}, we utilize the MATLAB 
solver \texttt{fmincon} with ``MaxIterations$=10^4$'' and opt for the 
interior-point method as an algorithm that runs within 
\texttt{fmincon}. %\footnote{One may think of using the SQP method in place of 
%the interior-point method. In the preliminary numerical experiments,  
%\texttt{fmincon} using the interior point method (fminconIPM) performed better 
%than the one with the SQP method (fminconSQP).Indeed, we observed that 
%fminconIPM solved each smoothed subproblem for \eqref{al:0913} more stably 
%than 
%fminconSQP.We also observed that fminconIPM was more likely to find a positive 
%solution $\lambda_1$, leading to a sparse solution. This result may be 
%peculiar 
%to an interior point method that reaches a solution from the interior of a 
%feasible region. %, leading to that }
%Solutions which are output by \texttt{fmincon} according to its default stopping criteria are employed as iterates.
 We initialize \texttt{fmincon} for \eqref{one-levelproblem} at some initial 
 point $(\w^0,\lambda^0)$ in the 
 first iteration $k=0$, and then use the previous iteration point 
 $(\w^{k-1},\lambda^{k-1})$ as the initial point for the succeeding iterations, 
 i.e., for $k\ge 1$. The 
 smoothing parameter is initialized at $\mu_0=0.1$, and the factor of decrease 
 is set to $\beta_1 = 0.8$. To obtain a reasonable initial point 
 $(\w^0,\lambda^0)$, we employ a semismooth Newton (SSN) method for solving the 
 KKT 
 system \eqref{KKT1}-\eqref{KKT5}. We first use a 
 complementarity function to  
 reformulate the conditions \eqref{KKT5} with $\varepsilon_5 = 0$ as a system 
 of equations 
 \cite{ALNCC20}. In particular, we use the Fischer-Burmeister function 
 $\phi_{\rm FB} : \Re^r \to \Re^r$ given by 
 	\[\phi_{\rm FB} (x,y) = x + y - \sqrt{x^2 + y^2},\]
 where the operations are understood to be taken component-wise, so that the 
 conditions \eqref{KKT5} with $\varepsilon_5 = 0$ are equivalent to 
 solving
 	\begin{equation}
 	\phi_{\rm FB} (\lambda-\epsilon e_1, \eta ) = 0.
 	\label{eq:fb}
 	\end{equation}
 With this, a KKT point satisfying \eqref{KKT1}-\eqref{KKT5} can be obtained 
 by solving approximately the equation  
 	\begin{equation}
 	\Phi^{\mu}_{\rm FB} (\w,\lambda,\zeta,\eta) \coloneqq  
 	\left( 
 	\begin{array}{c}
 	\nabla f(\omega) + (\nabla^2_{\omega\omega}G(\omega, \bar{\lambda}) + 
 	\lambda_1\nabla^2\varphi_{\mu}(\omega))\zeta \\
 	\nabla\varphi_{\mu}(\omega)^T \zeta - \eta_1 \\
 	\nabla \bar{R}(\w)^T \zeta - \bar{\eta} \\
 	\nabla_{\omega}G(\omega, \bar{\lambda}) + 
 	\lambda_1\nabla\varphi_{\mu}(\omega) \\
 	\phi_{\rm FB} (\lambda - \epsilon e_1,\eta) 
 	\end{array}\right) = 0.
 	\label{eq:Phi_FB}
 	\end{equation}
By our differentiability assumptions on $f$, $g$, $R_j$ ($j=2,\dots,r$) and the 
smoothness of $\varphi_{\mu }$, equations \eqref{KKT2}-\eqref{KKT4} are all 
smooth. On the other hand, from equations \eqref{phi_secondder} and 
\eqref{Hessian-matrix} and invoking Assumptions~(A) and (B), equation 
\eqref{KKT1} is semismooth provided that $\rho$ is semismooth, which is the 
case for piecewise smooth functions \cite[Proposition 7.4.6]{FP03}, such as the 
density functions that we considered in Appendix~\ref{app:examples_smoothing}. 
Finally, since $\phi_{\rm FB}$ is strongly semismooth \cite[Proposition 
7.4.8]{FP03}, then equation \eqref{eq:fb} is likewise a semismooth equation. 
Hence, we may use the semismooth Newton method for solving \eqref{eq:Phi_FB} for a fixed $
\mu$. 
Our warmstarting algorithm to obtain an initial point $(\w^0,\lambda^0)$ is 
described in Algorithm~\ref{alg:SSN}.\footnote{One may consider employing the 
semismooth Newton method, i.e., Algorithm~\ref{alg:SSN}, as a standalone 
algorithm for obtaining BKKT points. However, proceeding in this manner does 
not always lead to an accurate solution, primarily since this approach is only 
guaranteed to be locally convergent. Another particular hurdle is that apart 
from providing an initial guess for primal variables $(\w^0,\lambda^0)$, we 
also need to provide an initial guess for the Lagrange multipliers 
$(\zeta^0,\eta^0)$ when using Algorithm~\ref{alg:SSN}, which could influence 
the quality of the solution obtained by the algorithm. } Similar to Algorithm~\ref{algorithm}, we consider a sequence of equations \eqref{eq:Phi_FB} for decreasing values of $\mu$. In our 
experiments, we 
set 
$\tau_1 = 
0.8$, $\tau_2 = 0.1$, $\gamma_{\min} = 
10^{-6}$, $\mu_0 = 10$, and $\gamma_0=0.1$. We initialize 
Algorithm~\ref{alg:SSN} with $\w^0 = 100*\texttt{ones}(n,1)$ and $\lambda^0 = 
(\epsilon,0)$, $\zeta^0=0$, and $\eta^0=0$.
%, and $\eta^0$ is chosen so that the second and third 
%equations of \eqref{eq:Phi_FB} are satisfied, i.e., $\eta_1^0 \coloneqq 
%\nabla\varphi_{\mu_0}(\omega^0)^T \zeta^0$ and $\bar{\eta}^0\coloneqq \nabla 
%\bar{R}(\w^0)^T 
%\zeta^0$. 

\begin{algorithm} 
	\caption{(A semismooth Newton method for solving the bilevel KKT 
	system)}\label{alg:SSN}
	%\begin{algorithmic}[1]
	\begin{description}
		\item [Step 0] Choose $\mu_0, \gamma_0> 0$, $\tau_1,\tau_2, 
		\gamma_{\min} \in 
		(0,1)$, and $z^0 \coloneqq 
		(\w^0,\lambda^0,\zeta^0,\eta^0)$. Set 
		$k\coloneqq0$.
		
		\item [Step 1] Select an element $V^k \in \partial_C \Phi^{\mu_k}_{\rm 
		FB}(z^k)$, where $\partial_C$ denotes the Clarke subdifferential (see 
		\cite[Definition 7.1.1]{FP03}), and solve the linear system
			\[\Phi^{\mu_k}_{\rm FB}(z^k) + V^k \Delta z^k = 0.\]
		
		\item[Step 2] Set $z^{k+1} \coloneqq z^k + \Delta z^k$. 
		
		\item [Step 3] Set
			\[(\mu_{k+1}, {\gamma_{k+1}}) \coloneqq \begin{cases}
			(\mu_k, {\gamma}_k) & \text{if } \|\Phi^{\mu_k}_{\rm 
			FB}(z^k)\|  > {\gamma}_k \\
			(\tau_1\mu_k, \min \{ \tau_2{\gamma}_k, \gamma_{\min}\} ) & 
			\text{if } 
			\|\Phi^{\mu_k}_{\rm FB}(z^k)\| \leq {\gamma}_k
			\end{cases}\]
		If $\|\Phi^{\mu_k}_{\rm FB}(z^k)\| < \gamma_{\min}$, terminate the 
		algorithm. Otherwise, go to Step 1 and set $k\coloneqq k+1$.
	\end{description}
	%\end{algorithmic}
\end{algorithm}
%To compute the value of the function ${\rm erf}$ in $\phi_4$, we utilize the 
%MATLAB function \texttt{erf}.

\medskip 
In light of the SB-KKT conditions \eqref{SBKKT1}-\eqref{SBKKT6} and the value of the smoothing parameter $\mu$, we terminate the algorithm when either one of the following criteria is met:
\begin{enumerate}
\item The norms of the residuals of the equations 
in \eqref{bkkt1}-\eqref{bkktlast} are smaller than $10^{-2}$.
%As for condition\,\eqref{bkktlast}, we observe only the value of 
%$(\lambda^{\ast})^{\top}\eta^{\ast}$ because the nonnegativity of
%$(\lambda^{\ast},\eta^{\ast})$ is assured in the interior-point method. 
To estimate the index set $I(\w^{\ast})$ in conditions\,\eqref{SBKKT3} and 
\eqref{SBKKT4}, we regard $\w^k_i$ as zero if $\lvert \w^k_i\rvert\le 10^{-5}$.
\item $\mu_{k+1}\le 10^{-8}$. 
\end{enumerate}
%% taken from \2ed{the} UCI machine learning repository\,\cite{Lichman:2013}.
%% See Table~\ref{dataset} for the data sets actually used.  
%% For each data set, Algorithm~\ref{alg:smoothing} and \texttt{bayesopt} were applied to \eqref{al:0913}, respectively.}

The obtained results are summarized in Tables\,\ref{ta1}-\ref{ta4}, in which each column is described as follows.
Here, the averages are taken over the set of problems that are counted in success($\%$). 

\begin{table}[h]
\begin{center}
\begin{minipage}{\textwidth}
  \begin{tabular}{crl}
&$i$: &the smoothing function $\phi_i$\\  
&{val}: & average validation error at the resulting solution; the validation \\
& &  error is the least squares error for the validation data $A_1$ at the \\
& & obtained solution, i.e., the value of the objective function at the \\ 
& & resulting solution\\
&{test}: & average test error at the resulting solution; the test error is the\\
& &   the least squares error for the validation data $A_3$ at the obtained \\
& & solution, i.e., the value of the objective function at the resulting\\ 
& &  solution\\
&{bkkt}: & average residual of the BKKT conditions\\
%    &$\mu_{\rm end}$: &averaged value of $\mu_{k+1}$ on termination\\
    &${\rm sparsity}(\%)$: &average ratio of zero elements of the resulting 
    solution $\w^*$,\\
    &\ \ &in which each element $w_i$ is counted as zero if $\lvert 
    \w_i\rvert\le 10^{-5}$\\
&{time(s)}: & average time spent by the algorithm in seconds; in parenthesis, 
we \\
& & include the average time spent in the initialization phase via \\ 
& & Algorithm 
\ref{alg:SSN}\\
&{ssn.iter}: & average number of iterations for the initialization phase\\
&iter: & average number of iterations of Algorithm~\ref{algorithm} 
executed by employing \\ 
& & Matlab's fmincon built-in function. \\
    &{success($\%$)}: &ratio of problems for which BKKT points are computed 
    successfully\\
    &\ \ &in the sense that the first termination condition in the above is \\
    & &     satisfied
  \end{tabular}
\end{minipage}
\end{center}
\end{table}

%% In the tables, the results with the smoothing functions $\phi_4$ and $\phi_5$
%% are excluded. This is because as for $\phi_4$,
%% it tends to spend much more time
%% than the others, particularly on computing the integral function ${\rm erf}(\cdot)$ every time $\phi_4(x)$ is evaluated (we used the MATLAB command \texttt{integral} to compute the integral value). Indeed, for solving one instance with
%% $(n,m_1,m_2,p)=(500,1000,1000,0.5)$, the cputime spent is more than 800 seconds.
The best values are displayed 
in bold in the tables,  
%In addition, the smallest ones for $\mu_{\rm end}$ and ite are also in bold. 
with the results for the smoothing function $\phi_5$ excluded from the tables
%% $(n,m_1,m_2,p)=(500,1000,1000,0.5)$, the cputime spent is more than 800 seconds.
due to the overflow that often occurred when computing its gradient as $\mu$ 
gets smaller. % and it did not work well.
%\tred{We firast compare the results with $m_{\tr}=300$ and $n$}
%\tred{[[In light of the new experimental results, we need to revise the following observational facts.]]}
Now, the following insights are obtained from the numerical results. 

\paragraph{\alertsection{Comparison with $p=0.5,1$ and $m_2 = 300, 1000$}}
In terms of the sparsity of solutions obtained, we see that $\ell_{0.5}$ tends 
to attain sparser solutions than $\ell_{1}$. Indeed, it is evident from Table 2 
(resp., Table 4) that the solutions obtained for $p=0.5$ are sparser than those 
obtained by $p=1$ shown in Table 1 (resp., Table 3).  
%For example, by Tables\,\ref{ta3} and \ref{ta4}, ${\rm sparsity(\%)}$ for 
%$p=1$ takes values less than 10\%, whereas they are more than 60\% for 
%$p=0.5$. 
This is by virtue of the nonconvexity of $\ell_{p}$ with $p<1$.
Moreover, $\ell_{0.5}$ tends to attain solutions with better validation errors 
than $\ell_1$. 
%Indeed, values of val in Table\,\ref{ta3} with $p=1$ are around 
%$3.5\cdot10^{-3}$, while those in Table\,\ref{ta4} with $p=0.5$ are around 
%$7.2 
%\cdot 10^{-4}$. Similar observations can be gathered from Tables\,\ref{ta1} 
%and 
%\ref{ta2}.

%\paragraph{Comparison of results with $m_2=300,1000$:}
On the other hand, the problems with $m_2<n$ is related to the problem of 
finding sparse solutions of underdetermined linear systems.
Such kind of problems are often regarded more intractable than those with 
$m_2\ge n$, as illustrated by the obtained numerical results. When $p=1$, the 
success rate of the smoothing algorithm is largely diminished when $m_2<n$. In 
addition, it is clear from Tables \ref{ta1} and \ref{ta3} that for this 
instance, the algorithm required more time to solve the problems as compared 
when $m_2\geq n$. When $p<1$, while the average times spent by the algorithm are 
apparently not very distinct for both $m_2=1000$ and $m_2=300$, it is evident 
from Tables \ref{ta2} and \ref{ta4} that the success rate is also diminished 
for the latter case. Meanwhile, for the case $m_2<n$, we note that the success 
rate when $p=0.5$ is significantly better than when $p=1$.

\paragraph{\alertsection{Comparison of the four smoothing functions}}
In view of the validation and test errors, bilevel KKT residuals, sparsity, average time 
and success rates,
the qualities of the resulting solutions as well as the efficiency of the 
algorithm with different smoothing functions are comparable. From Table 1, we 
see that Algorithm~\ref{algorithm} with $\phi_4$ is the fastest method 
obtaining a 100$\%$ success rate in solving the problems, but the solutions 
obtained are neither the sparsest ones, nor do they correspond to the lowest 
validation and test errors. As these factors are quite important in evaluating the 
performance of the model, we observe that Algorithm~\ref{algorithm} equipped 
with smoothings functions $\phi_1$ and $\phi_2$ provide higher quality of 
solutions attained at a running time not significantly longer than that 
required by $\phi_4$. Considering these important criteria along with the 
success rates of the algorithms, we also observe from Tables 
\ref{ta2}-\ref{ta4} that the algorithm equipped with $\phi_1$ consistently 
obtains the best success rates with low validation error, as well as sparser 
solutions.

\begin{table}[h]
	\begin{center}
		\begin{minipage}{\textwidth}
			% \centering 
			\caption{Averaged results for $(n,m_2,m_1,p ) = (500, 1000, 
				1000, {1} )$} 
			\begin{tabular}{|c|c|c|c|c|c|c|c|c|}\hline
				$i$ & val & test &  bkkt & sparsity(\%) & time(s) & ssn.iter & 
				iter & success(\%) \\ \hline \hline 
				1 & \bftab 7.32e-03 & \bftab 7.31e-03 & 4.98e-03 & \bftab 45.8 
				& 187.5 (7.2) & 
				\bftab 182.1 & 34.9 & \bftab 100 \\ \hline 
				2 & 7.90e-03 & 7.83e-03 & 5.34e-03 & 43.4 & 175.5 (7.2) & 183.7 
				& 32.5 & \bftab 
				100 \\ \hline 
				3 & 9.17e-03 & 9.07e-03 & \bftab 4.23e-03 & 36.3 & 157.4 (7.2) 
				& 183.7 & 27.7 & 
				95 \\ \hline 
				4 & 1.08e-02 & 1.06e-02 & 4.59e-03 & 29.1 & {\bftab 148.9} 
				(7.1) & 185.2 & 
				{\bftab 26.4} 
				& 
				\bftab 
				100 \\ \hline 
				%5 & 1.17e-02 & 1.15e-02 & 4.81e-03 & 23.9 & {\bftab 120.5} 
				%(7.3) & 185.4 & 
				%\bftab 22.1 & \bftab 100 \\ \hline 
				6 & 1.01e-02 & 9.95e-03 & 4.80e-03 & 31.4 & 165.7 (7.2) & 182.9 
				& 30.9 & \bftab 
				100 \\ \hline 
			\end{tabular}
			\label{ta1}
		\end{minipage}
	\end{center}
\end{table}\hfill

%experiment 7
\begin{table}[h]
	\begin{center}
		\begin{minipage}{\textwidth}
			% \centering 
			\caption{Averaged results for $(n,m_2,m_1,p ) = (500, 1000, 
				1000, 0.5 )$} 
			\begin{tabular}{|c|c|c|c|c|c|c|c|c|}\hline
				$i$ & val & test &  bkkt & sparsity(\%) & time(s) & ssn.iter & 
				iter & success(\%) \\ \hline \hline 
				1 & 1.39e-03 & 1.39e-03 & \bftab 5.76e-03 & \bftab 84.7 & { 
				\bftab 198.9} 
				(7.5) & 180.9 & \bftab 31.2 
				& 
				\bftab 100 \\ \hline 
				2 & 1.36e-03 & 1.37e-03 & 6.66e-03 & \bftab 84.7 & 206.8 (7.4) 
				& 181.7 & 32.8 & 
				\bftab 100 \\ \hline 
				3 & \bftab 1.35e-03 & \bftab 1.36e-03 & 6.38e-03 & 83.5 & 213.3 
				(8.0) & 180.8 & 
				33.0 & \bftab 100 \\ \hline 
				4 & 3.01e-03 & 2.95e-03 & 6.96e-03 & 75.8 & 210.9 (7.6) & 183.3 
				& 32.5 & \bftab 
				100 \\ \hline 
				%5 & 1.67e-02 & 1.64e-02 & \bftab 4.03e-03 & 0.6 & {\bftab 
				%13.2} (7.2) & 183.3 
				%& 
				%\bftab 1.1 & 35 \\ \hline 
				6 & 3.02e-03 & 2.96e-03 & 6.05e-03 & 76.4 & 202.2 (7.4) & 
				\bftab 180.4 & 32.4 & 
				\bftab 100 \\ \hline 
			\end{tabular}
			\label{ta2}
		\end{minipage}
	\end{center}
\end{table}\hfill

%Experiment 1
\begin{table}[h]
	\begin{center}
		\begin{minipage}{\textwidth}
			% \centering 
			\caption{Averaged results for $(n,m_2,m_1,p ) = (500, 300, 
				300, 1 )$} 
			\begin{tabular}{|c|c|c|c|c|c|c|c|c|}\hline
				$i$ & val & test & bkkt & sparsity(\%) & time(s) & ssn.iter & 
				iter & success(\%) \\ \hline \hline 
				1 & 1.07e-02 & 1.10e-02 & 4.43e-03 & 58.6 & {\bftab 244.5} 
				(4.9) & 124.8 & 
				\bftab 44.4 & \bftab 50 \\ \hline 
				2 & \bftab 9.68e-03 & \bftab 1.02e-02 & 4.56e-03 & \bftab 59.9 
				& 268.6 (4.7) & 
				116.7 & 45.7 & 30 \\ \hline 
				3 & 1.23e-02 & 1.28e-02 & \bftab 3.85e-03 & 56.2 & 262.0 (5.2) 
				& 132.9 & 45.1 & 
				35 \\ \hline 
				4 & 1.17e-02 & 1.23e-02 & 5.32e-03 & 58.4 & 314.1 (6.4) & 147.2 
				& 52.7 & \bftab 
				50 \\ 
				\hline 
				%5 & 1.24e-02 & 1.25e-02 & 6.19e-03 & 58.1 & 296.5 (5.4) & 
				%141.5 & 56.4 & 
				%\bftab 
				%65 \\ \hline 
				6 & 2.24e-02 & 2.27e-02 & 7.66e-03 & 58.5 & 299.5 (4.4) & 
				\bftab 114.5 & 56.0 & 
				20 \\ \hline 
			\end{tabular}
			\label{ta3}
		\end{minipage}
	\end{center}
\end{table}\hfill	

%Experiment 2

\begin{table}[h]
	\begin{center}
		\begin{minipage}{\textwidth}
			% \centering 
			\caption{Averaged results for $(n,m_2,m_1,p ) = (500, 300, 
				300, 0.5 )$} 
			\begin{tabular}{|c|c|c|c|c|c|c|c|c|}\hline
				$i$ & val & test & bkkt & sparsity(\%) & time(s) & ssn.iter & 
				iter & success(\%) \\ \hline \hline 
				1 & 2.11e-03 & 2.25e-03 & \bftab 5.62e-03 & \bftab 80.8 & 185.3 
				(4.9) & \bftab 
				114.5 & 33.5 & \bftab 100 \\ \hline 
				2 & \bftab  2.06e-03 & 2.19e-03 & 6.00e-03 & 73.0 & {\bftab 
				180.6} (4.7) & 
				129.2 & 
				\bftab 33.2 & 85 \\ \hline 
				3 & 2.11e-03 & 2.20e-03 & 5.88e-03 & 75.2 & 189.5 (5.2) & 119.5 
				& 35.2 & 90 \\ 
				\hline 
				4 & \bftab 2.06e-03 & \bftab 2.15e-03 & 5.68e-03 & 79.2 & 200.9 
				(6.4) & 146.1 & 
				37.5 & 95 \\ \hline 
				%5 & 3.64e-03 & 3.57e-03 & 9.18e-03 & 77.2 & 244.8 (5.4) & 
				%164.8 & 43.5 & 20 \\ 
				%\hline 
				6 & 2.58e-03 & 2.58e-03 & 5.90e-03 & 75.2 & 196.6 (4.4) & 120.5 
				& 36.4 & 90 \\ 
				\hline 
			\end{tabular}
			\label{ta4}
		\end{minipage}
	\end{center}
\end{table}\hfill

\subsection{\alertsection{Problems with other regularizers}}
In this section, we solve problem\,\eqref{al:0913} with the regularizers 
$\psi_2(\|\w\|_p^p)$ and $\psi_3(\|\w\|_p^p)$ in place of $\|\w\|_p^p$, with
$\psi_2$ and $\psi_3$ defined in Appendix~\ref{app:penaltyfunctions}, where we 
set $a=1$ and $p=0.5$. 
%$(n,m_{\rm tr},m_{\rm val})=(500,1000,1000)$.  
%Moreover, we select $\phi_1$ and $\phi_6$ as smoothing functions in view of 
%the observation gained in the previous experiment. 
Both the experimental settings and the 20 synthetic problem-data of $A_i, b_i\ 
(i=1,2,3)$ are identical to the ones used in the preceding section.  The obtained 
results are summarized in Tables~\ref{ta5}-\ref{ta8}.

Similar to the remarks in the preceding sections, we observe that taking into 
account the quality of the solutions obtained as reflected by the validation 
errors and sparsity, together with the running times and success rates of the 
algorithm, we observe that Algorithm~\ref{algorithm} with the smoothing 
function $\phi_1$ has a consistent good performance among all the functions 
considered. 
%In both the cases, more than 95\% instances are successfully solved.
%
%In both Tables\,\ref{ta5} and \ref{ta6}, we observe that $\mu_{\rm end}$ for 
%$\phi_6$ are the smallest among the five smoothing functions, respectively, as 
%well as the results in Table~\ref{ta4} with the $\ell_{0.5}$ regularizer. 
%Moreover, those of $\phi_1$ are almost the largest although 
%the exactly largest one in Table~\ref{ta5} is attained by $\phi_2$.
%However, unlike the $\ell_{0.5}$ regularizer, superiority of $\phi_4$ in 
%time(s) is not observed.
%In comparison of Tables\,\ref{ta4}, \ref{ta5}, and \ref{ta6},
%the $\ell_{0.5}$ regularizer tends to gain sparser solutions according to 
%${\rm sparsity}(\%)$. 
%On the other hand, from ave-time(s), we see that the spent time for $\psi_3$ 
%tends to be the smaller than those for the $\ell_{0.5}$-regularizer and 
%$\psi_2$.

%Experiment 8
\begin{table}[h]
	\begin{center}
		\begin{minipage}{\textwidth}
			% \centering 
			\caption{Averaged results for $(n,m_2,m_1,p) = (500, 1000, 
				1000, 0.5)$ using $\psi_2$} 
			\begin{tabular}{|c|c|c|c|c|c|c|c|c|}\hline
				$i$ & val & test & bkkt & sparsity(\%) & time(s) & ssn.iter & 
				iter & success(\%) \\ \hline \hline 
				1 & 1.39e-03 & 1.39e-03 & \bftab 5.76e-03 & 84.7 & 276.7 (4.9) 
				& 
				180.9 & \bftab 31.2 & \bftab 100 \\ \hline 
				2 & 1.36e-03 & 1.37e-03 & 6.66e-03 & \bftab 84.7 & 284.4 (4.7) 
				& 181.7 & 32.8 & \bftab 100 \\ \hline 
				3 & \bftab 1.35e-03 & \bftab 1.36e-03 & 6.22e-03 & 84.1 & 259.4 
				(5.2) & 180.7 & 32.9 & 95 \\ \hline 
				4 & 3.01e-03 & 2.95e-03 & 6.96e-03 & 75.8 & 262.6 (6.4) & 183.3 
				& 32.5 & \bftab 100 \\ \hline 
				%				5 & 1.67e-02 & 1.64e-02 & \bftab 4.03e-03 & 0.6 
				%& {\bftab 19.9} 
				%				(5.4) & 183.3 & \bftab 1.1 & 35 \\ \hline 
				6 & 3.10e-03 & 3.05e-03 & 5.88e-03 & 76.0 & \bftab 244.1 (4.4) 
				& 
				\bftab 180.3 & 32.1 & 95 \\ \hline 
			\end{tabular}
			\label{ta5}
		\end{minipage}
	\end{center}
\end{table}\hfill

%Experiment 4
\begin{table}[h]
	\begin{center}
		\begin{minipage}{\textwidth}
			% \centering 
			\caption{Averaged results for $(n,m_2,m_1,p ) = (500, 300, 
				300, 0.5 )$ using $\psi_2$} 
			\begin{tabular}{|c|c|c|c|c|c|c|c|c|}\hline
				$i$ & val & test & bkkt & sparsity(\%) & time(s) & ssn.iter & 
				iter & success(\%) \\ \hline \hline 
				1 & \bftab 2.06e-03 & 2.17e-03 & 5.78e-03 & \bftab 80.7 & 200.4 
				(4.9) & \bftab 
				112.5 & \bftab 32.8 & \bftab 95 \\ \hline 
				2 & \bftab  2.06e-03 & 2.19e-03 & 6.00e-03 & 73.0 & 211.8 (4.7) 
				& 129.2 & 33.2 
				& 85 \\ 
				\hline 
				3 & 2.11e-03 & 2.17e-03 & 5.90e-03 & 74.9 & {\bftab 196.3} 
				(5.2) & 119.6 & 33.2 
				& 85 \\ \hline 
				4 & \bftab  2.06e-03 & \bftab 2.15e-03 & \bftab 5.68e-03 & 79.2 
				& 237.3 (6.4) & 
				146.1 & 
				37.5 & \bftab 95 \\ \hline 
				%5 & 3.64e-03 & 3.57e-03 & 9.18e-03 & 77.2 & 299.7 (5.4) & 
				%164.8 & 43.5 & 20 \\ 
				%\hline 
				6 & 2.58e-03 & 2.58e-03 & 5.90e-03 & 75.2 & 217.2 (4.4) & 120.5 
				& 36.4 & 90 \\ 
				\hline 
			\end{tabular}
			\label{ta6}
		\end{minipage}
	\end{center}
\end{table}\hfill

%Experiment 9
\begin{table}[h]
	\begin{center}
		\begin{minipage}{\textwidth}
			% \centering 
			\caption{Averaged results for $(n,m_2,m_1,p) = (500, 1000, 
				1000, 0.5)$ using $\psi_3$} 
			\begin{tabular}{|c|c|c|c|c|c|c|c|c|}\hline
				$i$ & val & test & bkkt & sparsity(\%) & time(s) & ssn.iter & 
				iter & success(\%) \\ \hline \hline 
				1 & 1.39e-03 & 1.39e-03 & \bftab 5.76e-03 & 84.7 & 193.1 (4.9) 
				& 180.9 & 
				\bftab 31.2 & \bftab 100 \\ \hline 
				2 & 1.37e-03 & 1.38e-03 & 6.68e-03 & \bftab 84.8 & 193.9 (4.7) 
				& 181.6 & 31.7 & 
				95 \\ \hline 
				3 & \bftab 1.35e-03 & \bftab 1.36e-03 & 6.38e-03 & 83.5 & 195.0 
				(5.2) & 180.8 & 
				33.0 & \bftab 100 \\ \hline 
				4 & 3.01e-03 & 2.95e-03 & 6.96e-03 & 75.8 & {\bftab 192.4} 
				(6.4) & 183.3 & 32.5 
				& \bftab 100 \\ \hline 
				6 & 3.02e-03 & 2.96e-03 & 6.05e-03 & 76.4 & 214.5 (5.4) & 
				\bftab 180.4 & 32.4 & 
				\bftab 100 \\ \hline 
			\end{tabular}
			\label{ta7}
		\end{minipage}
	\end{center}
\end{table}\hfill

%Experiment 5
\begin{table}[h]
	\begin{center}
		\begin{minipage}{\textwidth}
			% \centering 
			\caption{Averaged results for $(n,m_2,m_1,p ) = (500, 300, 
				300, 0.5 )$ using $\psi_3$} 
			\begin{tabular}{|c|c|c|c|c|c|c|c|c|}\hline
				$i$ & val & test &  bkkt & sparsity(\%) & time(s) & ssn.iter & 
				iter & success(\%) \\ \hline \hline 
				1 & 2.11e-03 & 2.25e-03 & \bftab 5.62e-03 & \bftab 80.8 & 
				227.6 (4.9) & \bftab 114.5 & 33.5 & \bftab 100 \\ \hline 
				2 & \bftab 2.06e-03 & 2.19e-03 & 6.00e-03 & 73.0 & 227.8 (4.7) 
				& 129.2 
				& \bftab 33.2 & 85 \\ \hline 
				3 & 2.11e-03 & 2.20e-03 & 5.88e-03 & 75.2 & 231.4 (5.2) & 119.5 
				& 35.2 & 90 \\ \hline 
				4 & \bftab 2.06e-03 & \bftab 2.15e-03 & 5.68e-03 & 79.2 & 
				{\bftab 220.9} (6.4) & 146.1 & 37.5 & 95 \\ \hline 
				%				5 & 3.64e-03 & 3.57e-03 & 9.18e-03 & 77.2 & 
				%264.5 (5.4) & 164.8 
				%				& 43.5 & 20 \\ \hline 
				6 & 2.58e-03 & 2.58e-03 & 5.90e-03 & 75.2 & 231.5 (4.4) & 120.5 
				& 36.4 & 90 \\ \hline 
			\end{tabular}
			\label{ta8}
		\end{minipage}
	\end{center}
\end{table}\hfill

\subsection{\alertsection{Comparisons with Bayesian optimization}}

{As mentioned in the introduction, two popular methods for dealing with the 
hyperparameter learning problem include the grid search method and Bayesian 
optimization. For practical purposes, however, grid search algorithm is not a 
viable approach due to the necessity of solving the lower level problem 
\eqref{lower-levelproblem0} for many values of the hyperparameters 
$(\lambda_1,\dots, \lambda_r)$, as was also demonstrated in \cite{OTKW21}. 
Hence, we only compare our approach with Bayesian optimization. As shown in 
Table \ref{ta9}, our approach needed only roughly 25$\%$ of the time required 
by Bayesian optimization for $p=0.5$, while still achieving low validation and 
test errors, as well as sparse models. For $p=1$, while the Bayesian 
optimization strategy attained sparser solutions, it required almost eight 
times more computing time, and the validation and test errors are significantly 
larger than the one obtained by our approach. }

\begin{table}[h]
	\begin{center}
		\begin{minipage}{\textwidth}
			 \centering 
			\caption{Averaged results for $(n,m_2,m_1) = (250, 500, 
				500)$ using $\psi_1$} 
			\begin{tabular}{|c|c|c|c|c|c|}\hline
				$p$ & method & val & test & sparsity(\%) & time(s) \\ \hline 
				\multirow{2}*{1} & Algorithm 1 w/ $\phi_1$ & \bftab 4.15e-03 & 
				\bftab 4.12e-03 & 41.7 & 
				\bftab 39.4 \\ 
				& Bayesian Optimization & 7.10e+00 & 6.79e+00 & \bftab 50.4 & 
				314.8 
				\\ \hline 
				\multirow{2}*{0.5} & Algorithm 1 w/ $\phi_1$ & \bftab 6.60e-04 
				& 
				\bftab 6.67e-04 & \bftab  72.7 
				& \bftab 39.9 \\ 
				& Bayesian Optimization & 4.47e+01 & 4.47e+01 & 19.1 & 160.8 
				\\ \hline  
			\end{tabular}
			\label{ta9}
		\end{minipage}
	\end{center}
\end{table}\hfill
%------------------------------------------------------------------------------------ Section 5
 \section{Conclusion}
 This paper considers a class of nonsmooth, possibly nonconvex and 
 non-Lipschitz regularizers for the best hyperparameter selection problem using 
 a bilevel programming strategy. The class of regularizers we consider subsumes 
 the traditional $\ell_p$ regularizer, which is the focus of the earlier work 
 \cite{OTKW21}. \alert{We propose new bilevel KKT conditions which are tighter than 
 the SBKKT conditions proposed in \cite{OTKW21}. These are necessary conditions 
 for the original bilevel problem \eqref{bileveloptimization} when $p=1$, and 
 are necessary conditions for the relaxed problem \eqref{first-orderoptimality} 
 when $p<1$.} The convergence analysis of the smoothing algorithm 
 presented in this paper is unified, in the sense that it is not limited to the 
 chosen smoothing function, unlike the previous work \cite{OTKW21} where the 
 analysis is centered on the selected smoothing function. Finally, we proved 
 our main convergence result under a milder constraint qualification. More 
 precisely, we only assumed the Mangasarian-Fromovitz constraint qualification 
 (MFCQ) for our convergence analysis, which is weaker than the linearly 
 independent constraint qualification (LICQ) used in \cite{OTKW21}. For our 
 numerical simulations, we compared the performance of six smoothing functions 
 in solving the bilevel programming problem using different regularizers. 
 Theoretically, we can use these smoothing functions for all the regularizers 
 considered as their corresponding density functions satisfy \alert{Assumption~(B)}. \alert{On 
 the other hand, our practical experience revealed that the smoothing function 
 $\phi_1$ provides the best performance when taking into account the validation 
 and test errors of the resulting model, as well as the sparsity of the 
 solution and running time of the algorithm. Interestingly, the function $\phi_1$ is the closest approximation to the regularizer $R_1$ among all the smoothing functions, as proved in Appendix~\ref{app:examples_smoothing}. }

 \medskip

%--------------------------------------------------------------------------
\section*{Acknowledgement}
Part of this work was conducted while J. H. Alcantara was a postdoctoral fellow 
at National Taiwan Normal University, and C. T. Nguyen was a research assistant 
at National Taiwan Normal University. The work of J.-S. Chen was supported by 
the Ministry of Science and Technology, Taiwan.

\begin{appendices}
%--------------------------------------------------------------------------
\section{Penalty functions that satisfy Assumption 
(A)}\label{app:penaltyfunctions}

We consider four penalty functions as follows:
 \[
 \psi_1(t)=t, ~~ \psi_2(t)= \log(1+at), ~~ \psi_3(t)=\frac{at}{1+at}, ~~ \psi_4(t)=\frac{-1}{1+at},
 \]
 where $a$ is positive number. In particular,
 \begin{description}
 \item [(1)] $\psi_1$ is a soft-thresholding penalty function \cite{HHM08,T96}. We have $\psi'_1(t)=1$ and $\psi''_1(t)=0$. Hence, it satisfies Assumption~(A).
 
 \item [(2)]  $\psi_2$ is a logistic penalty function \cite{NNZC08}. We have
 \[
 \psi'_2(t) = \frac{a}{1+at}, ~~ \psi''_2(t) = -\frac{a^2}{(1+at)^2},
 \]
 which implies that $0 < \lim_{t\to 0}\psi_2'(t) = a$ and $\lvert\psi''_2(t)\rvert \leq a^2$. Hence, it satisfies Assumption~(A).

 \item [(3)]  $\psi_3$ is fraction penalty function \cite{CZh10, NNZC08}. We have
 \[
 \psi'_3(t) = \frac{a}{(1+at)^2}, ~~ \psi''_3(t) = -\frac{2a^2}{(1+at)^3},
 \]
 which implies that $0 < \lim_{t\to 0}\psi_3'(t) = a$ and $\lvert\psi''_3(t)\rvert\leq 2a^2$. Hence, it satisfies Assumption~(A).

 \item [(4)] For function $\psi_4$, we have
 \[
 \psi'_4(t) = \frac{a}{(1+at)^2}, ~~ \psi''_4(t) = -\frac{2a^2}{(1+at)^3},
 \]
 which implies that $0 < \lim_{t\to 0}\psi_4'(t) = a$ and $\lvert\psi''_4(t)\rvert \leq 2a^2$. Hence, it satisfies Assumption~(A).
 \end{description}

%---------------------------------------------------------------------------------------
 \section{\alert{Examples of Smoothing Functions}}\label{app:examples_smoothing}
 A key aspect in successful numerical implementations of a smoothing algorithm is the choice of the approximating functions. Here, we enumerate six smoothing functions that we will use in our numerical simulations.

 %% There are many kernel (density) functions commonly used in statistics (see also \cite{CM96}). Some kernel functions satisfying \eqref{kernelcondition} are given as follows.
 There are many density functions commonly used and called kernel functions in statistics (see also \cite{CM96}). Some density functions satisfying \eqref{kernelcondition} are given as follows.
 \begin{eqnarray*}
 	\rho_1(x) &\coloneqq&
 	\left\{
 	\begin{array}{clc}
 		\frac{35}{32}(1-x^2)^3 & {\rm if} \; \lvert x \rvert \leq 1, \\
 		0 & \mbox{otherwise}.
 	\end{array}
 	\right.\\
 	\rho_2(x) &\coloneqq&
 	\left\{
 	\begin{array}{clc}
 		\frac{15}{16}(1-x^2)^2 & \mbox{if} \; \lvert x \rvert \leq 1, \\
 		0 &\mbox{otherwise}.
 	\end{array}
 	\right.\\
 	\rho_3(x) &\coloneqq&
 	\left\{
 	\begin{array}{clc}
 		\frac{3}{4}(1-x^2) & \mbox{if} \; \lvert x \rvert \leq 1, \\
 		0 &\mbox{otherwise},
 	\end{array}
 	\right.\\
 	\rho_4(x) &\coloneqq& \frac{1}{\sqrt{2\pi}}e^{-\frac{x^2}{2}}\quad \forall 
 	x \in \Re.\\
 	\rho_5(x) &\coloneqq& \frac{e^{-x}}{(1+e^{-x})^2}.\\
 	\rho_6(x) &\coloneqq& \frac{1}{(x^2 + \alert{1})^{\frac{3}{2}}}.\\
 \end{eqnarray*}
 
 Following the discussion in Section~\ref{subsec:smoothingfunctions}, the corresponding smoothing functions of $\lvert x \rvert$ are given as follows:
 \begin{eqnarray*}
 	\phi_1(\mu, x) &\coloneqq&
 	\left\{
 	\begin{array}{clc}
 		-\frac{5x^8}{128\mu^7} + \frac{7x^6}{32\mu^5} - \frac{35x^4}{64\mu^3} + \frac{35x^2}{32\mu} + \frac{35\mu}{128} & \mbox{if}\;  \lvert x \rvert \leq \mu, \\
 		\lvert x \rvert & \mbox{if}\; \lvert x \rvert > \mu.
 	\end{array}
 	\right. \label{SmoothF1}\\
 	\phi_2(\mu, x) &\coloneqq&
 	\left\{
 	\begin{array}{clc}
 		\frac{x^6}{16\mu^5} - \frac{5x^4}{16\mu^3} + \frac{15x^2}{16\mu} + \frac{5\mu}{16} & \mbox{if}\;  \lvert x \rvert \leq \mu, \\
 		\lvert x \rvert & \mbox{if}\; \lvert x \rvert > \mu.
 	\end{array}
 	\right. \label{SmoothF2}\\
 	\phi_3(\mu, x) &\coloneqq&
 	\left\{
 	\begin{array}{clc}
 		-\frac{x^4}{8\mu^3} + \frac{3x^2}{4\mu} + \frac{3\mu}{8} & \mbox{if}\;  \lvert x \rvert \leq \mu, \\
 		\lvert x \rvert & \mbox{if}\;  \lvert x \rvert > \mu.
 	\end{array}
 	\right. \label{SmoothF3}\\
 	\phi_4(\mu, x) &\coloneqq& x{\rm erf}\left(\frac{x}{\sqrt{2}\mu}\right)
 	+ \sqrt{\frac{2}{\pi}}\mu e^{-\frac{x^2}{2\mu^2}}, \label{SmoothF4}\\
 	\phi_5(\mu, x) &\coloneqq& \mu
 	\begin{bmatrix}
 		\log \begin{pmatrix}
 			1 + e^{-\frac{x}{\mu}} \\
 		\end{pmatrix} +
 		\log \begin{pmatrix}
 			1 + e^\frac{x}{\mu}\\
 		\end{pmatrix} \\
 	\end{bmatrix}.  \label{SmoothF5} \\
 	\phi_6(\mu, x) &\coloneqq& \sqrt{\alert{\mu^2} + x^2}. \label{SmoothF6}
 \end{eqnarray*}
 Here, the error function is defined by
 \[
 \mbox{erf}(x) = \frac{2}{\sqrt{\pi}}\int_0^x e^{-u^2}du \quad \forall x \in \Re.
 \] 
 
 \medskip
 \noindent The graphs of $\lvert x \rvert$ and $\phi_i(\mu,x)$, $i=1,2,\dots,6$ with $\mu=0.25$ are illustrated in Figure 1. From the graphs, we infer the following inequality relating the smoothing functions:
 \begin{equation*}\label{relation-phi}
\alertchieu{ \begin{cases}
 \lvert x \rvert \leq \phi_1(\mu, x) \leq \phi_2(\mu, x) \leq \phi_3(\mu, x) \leq \phi_4(\mu, x) \leq \phi_5(\mu, x), \phi_6(\mu, x).\\
 \mbox{there exists $\alpha>0$ such that}~ \phi_6(\mu, x) \leq \phi_5(\mu, x) ~ \mbox{for all}~ x\in [-\alpha,\alpha].
 \end{cases}}
 \end{equation*}
 \noindent It is not difficult to show that the relation $\lvert x \rvert \leq 
 \phi_1(\mu, x) \leq \phi_2(\mu, x) \leq \phi_3(\mu, x)$, while the proof of 
 the relation \alertchieu{$\phi_3(\mu, x) \leq \phi_4(\mu, x) \leq \phi_5(\mu, 
 x)$} can be found in \cite{SNC19}. \alertchieu{Using the same proof technique 
 in \cite{SNC19}, one can easily achieve the remaining inequalities}. On the 
 other hand, the graphs of the corresponding smoothing functions for $\lvert x 
 \rvert^p$ where $p\in (0,1]$ is shown in Figures \ref{fig:smfnc_p=1} and \ref{fig:smfnc_p<1}. \alert{We note that the} 
 smooth approximation $\phi_6$ is the 
 function used in \cite{OTKW21} for their smoothing algorithm for 
 \eqref{bileveloptimization} with $R_1(\omega) \coloneqq \sum_{i=1}^n 
 \lvert\omega_i\rvert^p$ $(0 < p \leq 1)$. 
 %----------------------------------------------------------------------------------------------- Figure 1
 
 \begin{figure}[!ht]
 	\centering
 	
 	\includegraphics[scale=0.5]{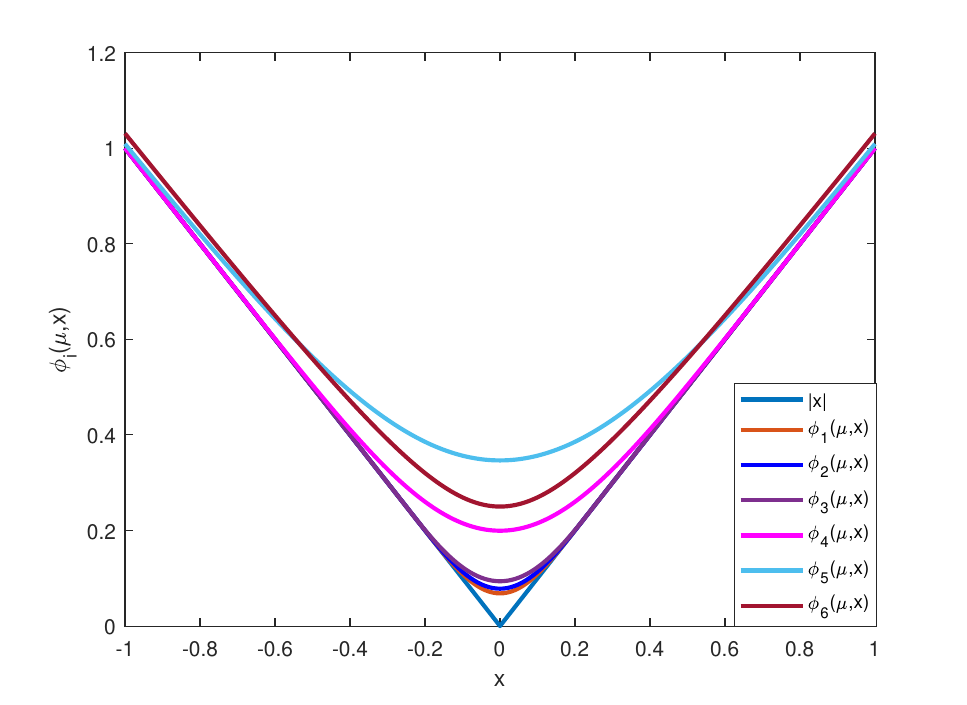}\\
 	\caption{Graph of $\lvert x \rvert$ and $\phi_i(\mu,x)$, $i=1,2,\dots,6$ with $\mu=0.25$.}
 	\label{fig:smfnc_p=1}
 \end{figure}

 \begin{figure}[!ht]
 	\centering
 	
 	\includegraphics[scale=0.5]{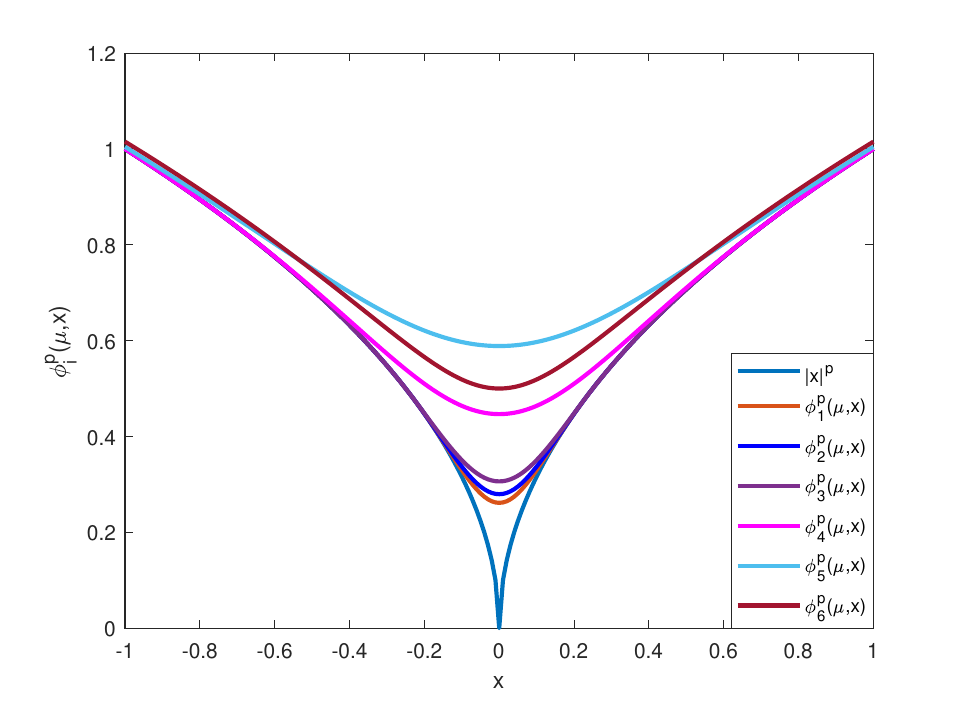}\\
 	\caption{Graph of $\lvert x \rvert^p$ and $(\phi(\mu,x))^p$, $i=1,2,\dots,6$ with $\mu=0.25$ and $p=0.5$.}
 	\label{fig:smfnc_p<1}
 \end{figure} 
 
 \medskip

% This function is simply $\phi_6 (\frac{\mu}{2},x)$ and can in fact be obtained by choosing the density function $\rho (t) = \frac{1}{2}(x^2+1)^{-3/2}$. 

 In this paper, we consider the six functions above and determine which 
 approximation is the best suitable in solving \eqref{bileveloptimization} with 
 $R_1(\omega)$ satisfying Assumption~(A). It is easy to check that the six 
 density functions as above satisfy Assumptions \alert{(B1)-(B2)}. Condition 
 (B3), on the other hand, holds by choosing $c=4$, and $r=2$. Indeed, 
 \[
 1-\frac{4}{4+S^2} \leq \sqrt{1-\frac{4}{4+S^2}}=2\int\limits_{0}^{S}\rho_3(s)~ds \leq 2\int\limits_{0}^{S}\rho_i(s)~ds ~~ \forall i=1,\dots,6.
 \]
 According to Assumption~(B4), only the functions $\rho_4$, $\rho_5$ and 
 $\rho_6$ can be used (theoretically) for the case $p=1$.
 
% \section{Initialization of Algorithm~\ref{algorithm}}
 
 \section{Proof of Lemma\,\ref{propertyMFCQ}}\label{sec:appB}
 In this appendix, we give a proof of Lemma\,\ref{propertyMFCQ}.
\begin{proof}
By Assumption~(C), we know that there exists $\bar{d}\in \Re^{n+r}$ such that 
\eqref{MFCQ1} and \eqref{MFCQ2} hold. Meanwhile, we have from the formula of 
$\Psi_j$ that 
	\begin{equation}
	\nabla _{(\omega,\lambda)} \Psi _j (\omega,\lambda) = \begin{cases}
	\nabla _{(\omega,\lambda)} \Phi_j (\omega,\lambda) & \text{if}~j\notin I(\omega^*) \\ 
	\left[ \begin{array}{c}
	e_j \\ 0_r
	\end{array}\right] & \text{if}~j\in I(\omega^*) , 
	\end{cases} \label{nablaPsi_j}
	\end{equation}
where $e_j$ is the $j$th standard unit vector in $\Re^n$ and $0_r$ denotes the 
zero vector in $\Re^r$. It is then clear from \eqref{nablaPsi_j} and 
\eqref{MFCQ1} that $\bar{d}_j = 0$ for all $j\in I(\omega^*)$. Consequently, 
letting $d\in \Re^{n-\lvert I(\omega^*)\rvert+r}$ be the vector $d\coloneqq 
(\bar{d})_{j\notin I(\omega^*)}$, it follows from \eqref{MFCQ1} and 
\eqref{MFCQ2} that equations \eqref{MFCQ1_1} and \eqref{MFCQ2_2} hold. 
\medskip 

It remains to show that $\{\nabla_{(\tomega,\lambda)}\Phi_j(\omega^*, 
\lambda^*)\}_{j\notin I(\omega^*)} $ is linearly independent. To this end, note 
first that we have from Assumption~(C) the linear independence of $\{ \nabla 
_{(\omega,\lambda)}\Psi _j (\omega^*,\lambda^*) \}_{j=1}^n$, that is, the 
matrix 
 \[
 M \coloneqq \left[\left(\nabla_{(\omega,\lambda)}\Psi_j(\omega^*, 
 \lambda^*)\right)_{(j\notin I(\omega^*))}, 
 \left(\nabla_{(\omega,\lambda)}\Psi_j(\omega^*, \lambda^*)\right)_{(j\in 
 I(\omega^*))}\right] \in \Re^{(n+r)\times n} 
 \]
 has full column rank. Using equation \eqref{nablaPsi_j} and switching the rows 
 of $M$ so that the first $\lvert I(\omega^*)\rvert$ rows correspond to the 
 index set $I(\omega^*)$, we have that the matrix 
 \[\begin{bmatrix}
 	(\nabla_{\hatomega}\Phi_j(\omega^*, \lambda^*))_{j\alert{\notin} 
 	I(\omega^*)}&E_{\lvert I(\omega^*)\rvert}\\ 
 	(\nabla_{(\tomega,\lambda)}\Phi_j(\omega^*, \lambda^*))_{j\notin 
 	I(\omega^*)}&O_{(n-\lvert I(\omega^*)\rvert + r)\times \lvert 
 	I(\omega^*)\rvert}\\ 
 \end{bmatrix}\]
 has full column rank, where $\hatomega\coloneqq (\omega_j)_{j\in 
 I(\omega^*)}$, $E_s$ 
 denotes the identity matrix of order $s$, and $O_{s\times t}$ is the zero 
 matrix of size $s\times t$. Since the upper and lower right blocks of the 
 above matrix are the identity matrix and zero matrix, respectively, a series 
 of elementary column operations leads us to conclude that 
  \[\begin{bmatrix}
 	O_{\lvert I(\omega^*)\rvert\times (n-\lvert I(\omega^*)\rvert)}&E_{\lvert I(\omega^*)\rvert}\\ 
 	(\nabla_{(\tomega,\lambda)}\Phi_j(\omega^*, \lambda^*))_{j\notin 
 	I(\omega^*)}&O_{(n-\lvert I(\omega^*)\rvert + r)\times \lvert 
 	I(\omega^*)\rvert}\\ 
 \end{bmatrix}\]
 also has full column rank. As a consequence, 
 $\{\nabla_{(\tomega,\lambda)}\Phi_j(\omega^*, \lambda^*)\}_{j\notin 
 I(\omega^*)}$ is linearly independent, as desired.
 \end{proof}

\end{appendices}


\end{document}